\documentclass[11pt,oneside,reqno]{amsart}
\usepackage{mathpazo} % math & rm
\normalfont
\usepackage[T1]{fontenc}

\usepackage[utf8]{inputenc}
\usepackage{amsmath}
\usepackage{amsfonts}
\usepackage{amssymb}
\usepackage{amsthm}
\usepackage{mathtools}%gives the arrow typeset for inclusion
\usepackage{stmaryrd}%Including this for the "double-bracket" \llbracket and \rrbracket
\usepackage{tikz}
\usepackage{tikz-cd}
\usepackage[all]{xy}
\usepackage{enumitem}%This package allows me to change the formatting of the numbers for an enumerate environment.
\usepackage[margin=1.2in]{geometry} % to set the margins

\usepackage{skak}
\usepackage{accents}

\usepackage{cjhebrew}%To use hebrew letters, because I ran out of letters
\usepackage{wasysym}%To use \gemini, because I ran out of letters

\usepackage{hyperref}

\newtheorem{thm}{Theorem}[section]
\newtheorem{theorem}[thm]{Theorem}
\newtheorem{lemma}[thm]{Lemma}

\newtheorem{coro}[thm]{Corollary}

\newtheorem{prop}[thm]{Proposition}
\newtheorem{proposition}[thm]{Proposition}

\theoremstyle{definition}
\newtheorem{defi}[thm]{Definition}
\newtheorem{definition}[thm]{Definition}
\newtheorem{remark}[thm]{Remark}
\newtheorem{example}[thm]{Example}
\newtheorem{notation}[thm]{Notation}
\newtheorem*{conjecture*}{Conjecture}

\newcommand\xqed[1]{%
  \leavevmode\unskip\penalty9999 \hbox{}\nobreak\hfill
  \quad\hbox{#1}}
\newcommand\exEnd{\xqed{$\diamondsuit$}}

\newcommand{\B}{\mathbb{B}}

\newcommand{\R}{\mathbb{R}}
\newcommand{\T}{\mathbb{T}}
\newcommand{\TT}{\T}
\newcommand{\Z}{\mathbb{Z}}
\newcommand{\ZZ}{\Z}
\newcommand{\N}{\mathbb{N}}
\newcommand{\Q}{\mathbb{Q}}

\newcommand{\calB}{\mathcal{B}}
\newcommand{\calb}{\calB}
\newcommand{\calC}{\mathcal{C}}
\newcommand{\calc}{\calC}
\newcommand{\calD}{\mathcal{D}}
\newcommand{\cald}{\calD}
\newcommand{\calF}{\mathcal{F}}
\newcommand{\calf}{\calF}
\newcommand{\calH}{\mathcal{H}}

\newcommand{\calN}{\mathcal{N}}

\newcommand{\calV}{\mathcal{V}}

\newcommand{\frakN}{\mathfrak{N}}
\newcommand{\frakP}{\mathfrak{P}}
\newcommand{\frakQ}{\mathfrak{Q}}
\newcommand{\frakV}{\mathfrak{V}}
\newcommand{\frakW}{\mathfrak{W}}

\newcommand{\inj}[0]{\ar@{^{(}->}}
\newcommand{\surj}[0]{\ar@{->>}}
\newcommand{\bij}[0]{\ar@{^{(}->>}}
\newcommand{\lbij}[0]{\ar@{_{(}->>}}
\newcommand{\linj}[0]{\ar@{_{(}->}}
\newcommand{\parr}[0]{\ar@{.>}}
\newcommand{\pinj}[0]{\ar@{^{(}.>}}
\newcommand{\psurj}[0]{\ar@{.>>}}
\newcommand{\pbij}[0]{\ar@{^{(}.>>}}
\newcommand{\lin}[0]{\ar@{-}}
\newcommand{\llin}[0]{\ar@{=}}
\newcommand{\usu}[0]{\ar@{->}}%The usual, unadorned, arrow. You can also just use \ar
\newcommand{\toup}[1]{\stackrel{#1}{\longrightarrow}}%An arrow to the right, taking an argument that is the name of the function, to be placed over the arrow
\newcommand{\onto}{\twoheadrightarrow}

\newcommand{\id}{\operatorname{id}}
\newcommand{\dsum}{\displaystyle\sum}
\newcommand{\dprod}{\displaystyle\prod}

\newcommand{\sdrop}{\backslash}
\newcommand{\into}{\hookrightarrow}
\newcommand{\eps}{\varepsilon}
\newcommand{\ph}{\varphi}
\newcommand{\supp}{\operatorname{supp}}%Support of, well, anything.
\newcommand{\Frac}{\operatorname{Frac}}%Fraction field of a integral domain
\newcommand{\Spec}{\operatorname{Spec}}%The spectrum of a ring.
\newcommand{\Terms}[1]{\operatorname{Terms}(#1)}%\Terms{S[\Mon]} is the set of terms in $S[\Mon]$
\newcommand{\trop}{\operatorname{trop}}
\newcommand{\Trop}{\operatorname{Trop}}

\def\Mon{\mathcal{M}}
\def\Lattice{M}
\def\base{S}%The base-semifield that we are working over. For now, it is S.
\newcommand{\Hom}{\operatorname{Hom}} 
\newcommand{\dlim}{\displaystyle\lim}
\newcommand{\dcap}{\displaystyle\bigcap}

\newcommand{\angbra}[1]{\left\langle #1\right\rangle}%Angle brackets, to be used for a pairing or a module being generated by a list of things.

\def\J-int{{J-integral}} 
\def\D-int{{D-integral}} 
\def\Q-int{{quasi-integral}}

\newcommand{\wt}[1]{\widetilde{#1}}
\newcommand{\what}[1]{\widehat{#1}}%A wide hat. Also, WHAT?
\newcommand{\conoidSet}{fan support set}

\newcommand{\conoidSets}{fan support sets}
\newcommand{\ConoidSets}{\conoidSets}

\newcommand{\ConoidLarry}{lattice of $\Gamma$-admissible \ConoidSets}

\newcommand{\wtcalf}{\wt{\calf}}%The white calf. That is, a prime filter of conoids 

\newcommand{\relint}{\operatorname{rel.int.}}
\newcommand{\rec}{\operatorname{rec}}%The recession cone of a polyhedron/recession of a fan/recession of a polyhedral set.
\newcommand{\Bend}{\operatorname{Bend}}%The bend congruence of an ideal.
\newcommand{\bend}{\operatorname{bend}}%The bend relations of a polynomial.
\newcommand{\vspan}{\operatorname{span}}%The span of a set of vectors. \span was already taken.

\newcommand{\cl}{\operatorname{cl}}%The closure of a set
\newcommand{\init}{\operatorname{in}}%Initial forms.
\newcommand{\whin}{\what{\init}}%The initial hat/duck of a polynomial with respect to a prime conrguence of a point.
\newcommand{\largeNum}{\calN}%A large number, possibly in a limit to infinity. We need to choose another name for it because we already have $N$ in $N_{\R}$.
\newcommand{\largenum}{\largeNum}
\newcommand{\idker}{\operatorname{ideal-kernel}}%The ideal-kernel of a congruence.

\newcommand{\dcup}{\displaystyle\bigcup}
\newcommand{\nbar}[1]{\overline{#1}}%A new bar. Better than the usual \bar.
\newcommand{\Carl}{\theta}%Carl is a cone contained in a boundary part of the closure of a fan support set in $\R_{\geq0}\times N_{\R}(\sigma)$ (in particular, a congruence variety).
\newcommand{\carl}{\Carl}

\DeclareFontFamily{U}{min}{}
\DeclareFontShape{U}{min}{m}{n}{<-> dmjhira}{}
\newcommand{\Simon}{\mathfrak{S}}%Simon is a simplex in a parameter space of $(\eps_1,\ldots,\eps_k)$.
\newcommand{\simon}{\Simon}
\newcommand{\Rachel}{\mathfrak{R}}%Simon's cousin
\newcommand{\rachel}{\Rachel}
\newcommand{\Paula}{\mathfrak{P}}%A polyhedron/polyhedral cone in a complete polyhedral complex/fan
\newcommand{\paula}{\Paula}
\newcommand{\Olive}{\mathfrak{O}}%Another polyhedron in a polyhedral complex, Paula's cousin
\newcommand{\olive}{\Olive}
\newcommand{\Tara}{\mathfrak{t}}%A set of cones
\newcommand{\tara}{\Tara}
\newcommand{\Valerie}{\zeta}%The image of a vertex under a linear map
\newcommand{\valerie}{\Valerie}
\newcommand{\Opal}{\Omega}%The image of an open simplex under a linear map

\newcommand{\CoordinateFunctionSemiring}{\operatorname{CPL}}
\newcommand{\CPL}{\CoordinateFunctionSemiring}
\newcommand{\CPA}{\operatorname{CPA}}

\title{Geometric classification of primes modulo a (bend) congruence% and a tropical Nullstellensatz-type result
}
\author{Netanel Friedenberg }
\address{Department of Mathematics, Tulane University, New Orleans, LA 70118, USA}
\email{nfriedenberg@tulane.edu}

\author[K.~Mincheva]{Kalina~Mincheva}
\address{Department of Mathematics, Tulane University, New Orleans, LA 70118, USA}
\email{kmincheva@tulane.edu}
\date{}

\begin{document}

\begin{abstract}

% \orange{
% In this paper we continue the program to develop the algebraic foundations of tropical (algebraic) geometry. We prove a Nullstellensatz-type statement for congruences with finite tropical basis, which are not necessarily finitely generated. \pinky{We show that, if $I$ is the ideal of an affine variety not contained in the coordinate hyperplanes, then $\T[x_1, \dots, x_n]/\sqrt{\Bend(\trop I)}$ is the tropical function semiring on $\trop V(I)$. Moreover, we prove a general result that implies that this semiring is cancellative by explicitly describing the minimal prime congruences containing $\Bend(I)$.} These results are consequences of a stronger structure result classifying and constructing minimal primes congruences over certain congruences satisfying geometric conditions. 

% We give three applications of this result: (1) It allows the study of total semiring of fractions, which has applications to integral closure in \cite{To16} and normalization in a forthcoming paper \cite{FM26}; (2) This creates a bridge between the bend relations (embedded) approach \cite{GG13} \cite{MR18} and the algebraic approach to non-embedded tropicalization in the work in \cite{Son23a}, \cite{Son23b} by relating the corresponding coordinate semirings; (3) We can describe the closure of a polyhedron in a tropical toric variety where the polyhedron is not necessarily compatible with the fan defining the tropical toric variety.
% }

In this paper we continue the program to develop the algebraic foundations of tropical (algebraic) geometry. We give strong characterizations of prime congruences containing a given congruence on a toric semiring. We give four applications of this result. (1) We prove an analogue of the strong Nullstellensatz for congruences with finite tropical basis. This extends the existing result of \cite{JM17} to cases, such as the bend congruence of a tropical(ized) ideal, where the congruence is not finitely generated. (2) We show that, if $I$ is the ideal of an affine variety not contained in the coordinate hyperplanes, then $\T[x_1, \dots, x_n]/\sqrt{\Bend(\trop I)}$ is cancellative. %This is a consequence of a stronger structure result which constructs minimal primes over congruences satisfying mild geometric conditions. 
This result has applications to the integral closure (as defined in \cite{To16}) of $\T[x_1, \dots, x_n]/\Bend(\trop I)$ which we explore in a forthcoming paper \cite{FM26}. (3) We show that $\T[x_1, \dots, x_n]/\sqrt{\Bend(\trop I)}$ is the tropical function semiring on $\trop V(I)$, which creates a bridge between the algebraic approach to non-embedded tropicalization in the work in \cite{Son23a} and \cite{Son23b} and the bend congruence approach of \cite{GG13} and \cite{MR18}. (4) As a consequence of our construction we describe the closure of a polyhedron in a tropical toric variety even when the polyhedron is not compatible with the fan defining the tropical toric variety.

\end{abstract}

\maketitle

%=======================================
% Notes (done with): 62
% Notes (left): 63-66
%=======================================

%\section{Conventions}
\section{Introduction}

Tropical geometry replaces algebraic varieties (the vanishing loci of polynomial equations) with tropical varieties (tropical vanishing loci of tropical polynomials), thus allowing for the use of a broader range of (combinatorial) tools in the study of algebraic varieties. While tropical geometry has been very successful in providing a new approach to classical problems, those applications have been focused on using the combinatorics and geometry and have ignored the underlying (semiring) algebra.

One way to endow tropical varieties (which are not necessarily tropicalizations of algebraic varieties) with algebraic structure is introduced in \cite{GG13} wherein the authors define a semiring analogue of the coordinate ring of a variety. As shown in \cite{MR14}, these coordinate semirings carry the same information as that of particular ideals, which the authors call tropical ideals, that are almost always not finitely generated. More precisely, let $\base$ be a sub-semifield of the tropical semifield $\T$, let $\Mon$ be a toric monoid, and let $I$ be any subset of $S[\Mon]$. The variety of $I$ is denoted $V(I)$ and is defined in terms of a ``tropical vanishing'' condition. Towards forming a corresponding coordinate semiring, we can form the \textit{bend congruence} $\Bend(I)$ of $I$; see Definition~\ref{def:bend_rel}. The coordinate semiring is then defined to be $\base[\Mon]/\Bend(I)$. We can then recover the variety of $I$ as $\Hom_{S-\mathrm{alg}}(\base[\Mon]/\Bend(I),\T)\cong V(I)$.

Coming from algebraic geometry, one may wonder why these congruences are necessary~\-- maybe we should just be quotienting by ideals, i.e., setting everything in the ideal equal to 0. If we do this, however, then affine space would only have finitely many subvarieties. Moreover, different ideals would have the same quotients.

%\orange{Motivation}
In analogue with classical algebraic geometry one would like to more deeply study the connection between tropical varieties and their coordinate semirings. 
Our main tool for doing so is the notion of a \emph{prime congruence}. A congruence on an additively idempotent semiring is called prime if the corresponding quotient is totally ordered and cancellative; see Definition~\ref{def: prime_domain}. In the context of polynomial or Laurent polynomial semirings, prime congruences are very concrete - they can be given explicitly by matrices.

A natural first question concerns a strong Nullstellensatz statement for congruences. At the moment one only exists for finitely generated congruences; see \cite{JM17}. However, if $I$ is a tropical ideal (or even the tropicalization of an ideal in a polynomial ring) then $\Bend(I)$ is almost never finitely generated. Similarly, at the moment there is also no strong Nullstellensatz for the corresponding tropical ideals.

One object that has recently been of interest is the semifield of functions on a tropical variety; see \cite{AKS25}, \cite{SN24}, \cite{Son24} etc. In the case where the variety is all of tropical affine space, Theorem 4.9 (iv), (v) and Theorem 4.14 (iv) in \cite{JM17} show that this is the same as the semifield of fractions of $\T[x_1,\ldots,x_n]/\sqrt{\Delta}$ where $\Delta$ is the trivial congruence and the radical $\sqrt{C}$ of a congruence $C$ is the intersection of all prime congruences containing $C$. However, no explicit algebraic presentation is known for the semifield of functions on any tropical variety other than tropical toric varieties. Moreover, it is not known whether $\T[x_1,\ldots,x_n]/\sqrt{\Bend(I)}$ is cancellative and so whether the total semiring of fractions of $\T[x_1,\ldots,x_n]/\sqrt{\Bend(I)}$ is even a semifield. 

In this paper we resolve both of these. In doing so, we also given an algebraic description of the congruences that are considered in \cite{Ito26}, where they are defined purely geometrically.

Our original motivation for this work was to be able to describe integral closure in the context of coordinate semirings. We continue this project in upcoming work.

\subsection{Results}   
We first introduce a modified version of a tropical vanishing locus. Let $\base$ be a sub-semifield of $\T$ and let $\Mon$ be a toric monoid corresponding to a (strictly convex rational polyhedral) cone $\sigma\subseteq N_{\R}$. Given $f=\dsum_{u\in\Mon}f_u\chi^u\in \base[\Mon]$ and $(r,x)\in \R_{\geq0}\times N_{\R}(\sigma)$, we let $\wt{f}(r,x)=\displaystyle\max_{u\in\Mon}\left(r\log(f_u)+_{\R}\angbra{x,u}\right)$. For any $E\subseteq (\base[\Mon])^2$ and $I$ ideal of $\base[\Mon]$ we define

\begin{equation*}
\begin{split}
%&V(E) = \{ a \in N_{\R}(\sigma) : f(a) = g(a), \text{ for all } (f,g) \in E\} \\
&\wt{V}(E) =\{w\in\R_{\geq0}\times N_{\R}(\sigma)\;:\;\wt{f}(w)=\wt{g}(w)\text{ for all }(f,g)\in E\}.\\
%&V(I) = \{a \in N_{\R}(\sigma) : \forall f \in I, f \text{ tropically vanishes at } a\}
\end{split}
\end{equation*}

This is a refinement of the ``usual'' congruence variety 
$V(E) = \{ a \in N_{\R}(\sigma) : f(a) = g(a), \text{ for all } (f,g) \in E\}$
that also takes into account the variety of the corresponding congruence with $\B$ coefficients. %\red{we don't say what B is..}
Here $\B\subseteq\T$ is the \emph{boolean semifield}, consisting of the semiring additive and multiplicative identity elements of $\T$.

If $I$ is the ideal of a subvariety $X$ of the torus and $E = \Bend(\trop I)$, the set $\wt{V}(E)$ exactly coincides with Gubler's tropicalization $\operatorname{Trop}_W(X)$ defined in \cite{Gub11}.

%\red{ Our first result gives a condition... }

%(Have we explained (1) prime congruences and (2) finite tropical basis?)

Our first result classifies exactly which prime congruences contain a given congruence $E$ under the mild hypothesis that $E$ has a finite tropical basis; c.f.\ Theorem~\ref{thm:PrimeContainsCongruence}. Here we say that a congruence $E$ has a finite tropical basis if its variety can be written as a finite intersection of varieties of a single pair $(f,g)\in E$; see Definition~\ref{def: finiteTropBasis}.
To state the result, we will use the concept of the prime congruence $P_{\calc_\bullet}$ defined by a flag of cones $\calc_\bullet$; see Definitions~\ref{def:flagOfCones} and~\ref{def:PrimeOfAFlag}.

\begin{theorem}\label{IntroThm:PrimeContainsCongruence}
Let $E$ be a congruence with a finite tropical basis and let $P$ be a prime congruence, both on $\base[\Mon]$.
Then $E\subseteq P$ if and only if there is a flag of cones $\calc_\bullet$ contained in $\wt{V}(E)$ such that $P=P_{\calc_\bullet}$. 
\end{theorem}

When $\Mon=\Z^n$ or $\N^n$, corresponding to the torus and to affine space, respectively, Theorem~\ref{IntroThm:PrimeContainsCongruence} says that $E\subseteq P$ if and only if there is a defining matrix for $P$ whose rows are all in $\wt{V}(E)$.

Our next result, Corollary~\ref{coro: infiniteNull}, is a generalization of the ``strong Nullstellensatz for congruences'' result of \cite[Theorem 5.4]{JM17} to the case when the congruence is not necessarily finitely generated but has a finite tropical basis. %Definition~\ref{def: trop_ideals} for the precise definition of a tropical ideal.

% \begin{theorem}
% Let $E$ be a congruence on $\base[\Mon]$ with a finite tropical basis $\calB$. Then the finitely generated congruence $C=\angbra{\calB}$ satisfies $\sqrt{E}=\sqrt{C}$. In particular, if $I\subseteq\base[\Mon]$ is a tropical ideal then there is a finitely generated ideal $J\subseteq I$ such that $\sqrt{\Bend(I)}=\sqrt{\Bend(J)}$. 
% \end{theorem}

\begin{theorem}[Nullstellensatz for congruences with a finite tropical basis]
Let $S \subseteq \T$ and $\Mon$ a toric monoid. Let $E$ be a congruence on $\base[\Mon]$ that has a finite tropical basis. Suppose $(f,g)$ is such that $\wt f(w)=\wt g(w)$ for all $w\in \wt V(E)$. Then there are $i\in\Z_{\geq0}$ and $h\in\base[\Mon]$ such that $\big[(f+g)^i+h\big](f,g)\in E$.
\end{theorem}

In particular, this conclusion holds when $E$ is the bend congruence of a tropical ideal; see Definition~\ref{def: trop_ideals} for the precise definition of a tropical ideal.

The next statement, Theorem~\ref{thm:resolving-primes}, is a strong structure result, which states that each minimal prime containing a congruence $E$ with finite tropical basis has trivial ideal-kernel. Here, the ideal-kernel of a congruence $E$ is the set of elements in the equivalence class of 0, the additive identity, in the quotient by $E$. Our proof is constructive and gives a recipe to produce those primes. This is a generalization of \cite[Theorem 4.9 and Theorem 4.14]{JM17} stating that minimal primes (over the diagonal) of a polynomial semiring have trivial ideal-kernel. The proof of our result is much more subtle and complicated - constructing a suitable prime and showing it actually contains $E$ requires significantly more nuance than doing so in the situation where the prime need only contain the diagonal.

\begin{theorem}\label{IntroThm:MinPrimes}
Let $E$ be a congruence on $\base[\Mon]$ with a finite tropical basis. 
Suppose that $\wt{V}(E)$ is the closure of $\wt{V}(E)\cap(\R_{\geq0}\times N_{\R})$ in $\R_{\geq0}\times N_{\R}(\sigma)$. Then every minimal prime congruence of $\base[\Mon]/E$ has trivial ideal-kernel.
\end{theorem}

The above statement implies in Corollary ~\ref{coro:ClosureConditionImpliesRadicalQC} that $\base[\Mon]/\sqrt{E}$, which in view of Section~\ref{app:CPL} can be called the function semiring of a well-behaved congruence variety, is cancellative.

\begin{theorem}\label{thm-intro-rad}
Let $\base$ be a sub-semifield of $\T$, let $\Mon$ be the toric monoid corresponding to the cone $\sigma$, and let $E$ be a congruence on $\base[\Mon]$ with a finite tropical basis. Suppose that $\wt{V}(E)$ is the closure of $\wt{V}(E)\cap(\R_{\geq0}\times N_{\R})$ in $\R_{\geq0}\times N_{\R}(\sigma)$. Then $\base[\Mon]/\sqrt{E}$ is cancellative.
\end{theorem}

In particular, we have the following corollary (see Remark~\ref{rmk:forTropicalizedCases}).

\begin{coro}\label{IntroCoro:TropIdeal}
Let $\base$ be a sub-semifield of $\T$, let $\Mon$ be the toric monoid corresponding to the cone $\sigma$, and let $I\subseteq\base[\Mon]$ be the tropicalization of an ideal in $k[\Mon]$ for some valued field $k\to\base$. Suppose that $V(I)$ is the closure of $V(I)\cap N_{\R}$ in the tropical toric variety $N_{\R}(\sigma)$. Then $\base[\Mon]/\sqrt{\Bend(I)}$ is cancellative.
\end{coro}

We suspect that Corollary~\ref{IntroCoro:TropIdeal} is also true when $I$ is a tropical ideal; see Remark~\ref{rmk:errorInMR}.

The hypothesis of a finite tropical basis in Theorem~\ref{thm-intro-rad} is essential; see Example~\ref{ex: inf-trop-basis} for an example where the result fails for a congruence that has no finite tropical basis.

%In Section~\ref{app:CPL} we give an alternative proof of Theorem~\ref{thm-intro-rad} but it does not construct the relevant primes and does not imply Theorem~\ref{thm:resolving-primes}. On the other hand, Corollary~\ref{coro: functions-know-radical} shows that $V(I)$ and $\sqrt{\Bend(I)}$ carry the same information. \orange{This proof presents a more geometric point of view while the structure theorem, Theorem~\ref{thm:resolving-primes}, approaches the problem from an algebraic point of view.} \red{state the section 6 result}

In addition to the proof of Theorem~\ref{thm-intro-rad}
using the algebraic approach of Theorem~\ref{IntroThm:MinPrimes}/Theorem~\ref{thm:resolving-primes}, we also give a geometric proof of this result in Section~\ref{app:CPL}. These geometric methods also give the following corollary (Corollary~\ref{coro: functions-know-radical}). Here $\CPA(X)$ means the semiring of $\Gamma$-rational piecewise affine functions on $X$.

\begin{coro}
Let $J\subseteq k[\Mon]$ be an ideal such that there is no affine toric variety that is a proper toric subvariety of $X(\sigma)$ and which contains $V(J)$. %\footnote{In particular, this is true if $V(J)$ meets every boundary divisor.}.
Let $I=\trop(J)\subseteq \base[\Mon]$. Then the evaluation map $\base[\Mon]\to\CPA\big(V(\Bend(I))\big)$ induces an isomorphism
$$\dfrac{\base[\Mon]}{\sqrt{\Bend(I)}}\cong\CPA\big(V(\Bend I)\big).$$
\end{coro}
In particular, this shows that $V(I)$ and $\sqrt{\Bend(I)}$ carry the same information.

%In Section~\ref{app:CPL} we give an alternative proof of Theorem~\ref{thm-intro-rad} but it does not construct the relevant primes and does not imply Theorem~\ref{thm:resolving-primes}. On the other hand, Corollary~\ref{coro: functions-know-radical} shows that $V(I)$ and $\sqrt{\Bend(I)}$ carry the same information. \orange{This proof presents a more geometric point of view while the structure theorem, Theorem~\ref{thm:resolving-primes}, approaches the problem from an algebraic point of view.} \red{state the section 6 result}

The above results allow us to express the rational function semifields of \cite{SN24}, \cite{Son24}, \cite{Son23a}, \cite{Son23b} in terms of quotients of polynomial semirings. This links the non-embedded tropicalization of those works with the bend congruence approach of \cite{GG13} and \cite{MR18}. We address this is Appendix~\ref{app:Eu-Ja}.

Our last result concerns the closure of a polyhedron in a tropical toric variety. Until now, such results were limited to cases where the polyhedron was compatible with the fan defining the tropical toric variety; see, for example, \cite[Definition 3.18 and Proposition 3.19]{Rab10}. Our Corollary~\ref{coro: toric_Application}  gives a tool for analyzing the closure of a polyhedron in a tropical toric variety for which the polyhedron is not necessarily compatible with the fan.

\begin{theorem}
Let $\Sigma$ be a fan in $N_{\R}$ and let $w\in N_{\R}(\Sigma)\sdrop N_{\R}$. If $\mathfrak{L}\subseteq N_{\R}$ is a polyhedron such that $w\in\cl_{N_{\R}(\Sigma)}(\mathfrak{L})$, then there are $\what{w}\in \mathfrak{L}$ and $v\in\operatorname{rec}(\mathfrak{L})$ such that $\dlim_{\largeNum\to\infty}\what{w}+\largeNum v=w$.
\end{theorem}

While this theorem is not particularly surprising, we do not see any very quick proof of it; see Remark~\ref{rmk: closures of polyhedra} for further discussion. 
% \orange{

% \newcommand{\Bd}{\operatorname{Bd}}

% While this theorem is not particularly surprising, we do not see any very quick proof of it. For example, one approach to proving this would be to reduce to the case where $\mathfrak{L}$ is full-dimensional and then consider cases as to whether $w$ is in the interior of $\cl_{N_{\R}(\Sigma)}\mathfrak{L}$ or not. If $w$ is in the interior then there is a basic open neighborhood of $w$ (see \cite[Remark 3.4]{Pay09} for the relevant neighborhood basis) that is completely contained in $\cl_{N_{\R}(\Sigma)}\mathfrak{L}$; using the structure of such a neighborhood it would then not be difficult to show that the relevant $\what{w}$ and $v$ exist inside this neighborhood. If $w$ is not in the interior, i.e., $w$ is in the boundary of $\cl_{N_{\R}(\Sigma)}\mathfrak{L}$, then one would hope to show that $w$ is in the closure of a proper face of $\mathfrak{L}$ and then use induction on dimension. However, in order to do this, one would need to show that $\Bd_{N_{\R}(\Sigma)} \left( \cl_{N_{\R}(\Sigma)} \mathfrak{L} \right) \subseteq \cl_{N_{\R}(\Sigma)} \left( \Bd_{N_{\R}}\mathfrak{L} \right)$ and, while intuitively clear, this does not seem to have a straightforward proof.

% Ultimately, our proof takes a different, case-free, approach which allows us to prove slightly more; see Lemma~\ref{lemma:PolyhedronClosureMeansEndOfRay}. The proof takes about one page. We bring this result to the attention of the community to save people time in not having to reproduce it or similar results from scratch.
% }  

\subsection{Context in the literature}

In recent years, there has been a lot of effort dedicated to developing the necessary tools for commutative semiring algebra to understand the geometry of tropical varieties; see \cite{GG13}, \cite{MR14}, \cite{MR18}, \cite{FGGJ24}, \cite{JM15}, \cite{JM17}, \cite{FM23}, \cite{SN24}, \cite{Son24}, \cite{Ito26}, \cite{BE13}, \cite{IR14}, \cite{GP14}, \cite{ABG24}. One particular approach, introduced in \cite{GG13}, focuses on constructing an analogue of a coordinate semiring which allows us to remember more information about the original algebro-geometric objects and are able to define tropical objects intrinsically.

%\orange{algebraic methods hype}
There are compelling results pointing to the power of semiring methods in tropical geometry. 

Tropicalizations of ideals in polynomial rings, or equivalently the bend congruences of such tropicalized ideals, have been shown to preserve more information than the tropicalizations of the corresponding varieties. For example, the tropical variety retains the information of the dimension and degree, whereas the whole Hilbert polynomial is determined by the tropical ideal; see \cite{GG13}, \cite{MR14}, and \cite{MR18}. Maclagan and the second author, in unpublished work, give concrete ways to use tropical ideals to distinguish planar quadric curves that have the same tropical variety.

A recent paper, \cite{AKS25}, focuses on the tropicalization of the function field of the variety over a discretely valued field and, extending a result of Baker and Rabinoff for curves, shows that this is the same as the semifield of tropical rational functions.
They also show that this semifield is finitely generated. These results lead them to new proofs of the discretely valued cases of a recent result by Ducros, Hrushovski, Loeser and Ye \cite{DHLY24} and of the faithful tropicalization theorem by Gubler, Rabinoff and Werner \cite{GRW16}.

In \cite{To16} Tolliver provides an application to number theory -- he gives a purely semiring-theoretic proof for the extension of valuations of fields. 

\section*{Acknowledgments} 
K.M. acknowledges the support of the Simons Foundation MPS-TSM-00008148. The authors would like to thank Diane Maclagan for her insightful comments on an earlier version of the paper.

\section{preliminaries}\label{prelims}

A \textit{semiring} is a set $R$ with two binary operations (addition $+$ and multiplication $\cdot$ ) satisfying the same axioms as rings except the existence of additive inverses, as well as the axioms that $a\cdot0=0$ and $0\cdot a=0$. In this paper, a semiring is always assumed to be commutative. A semiring $(R,+,\cdot)$ is a \emph{semifield} if $(R\backslash\{0_R\},\cdot)$ is a group. A semiring $R$ is called \emph{additively idempotent} if for all $a\in R$ we have that $a+a =a$.

\textbf{Key assumption:} All semirings will be assumed to be additively idempotent.

Any additively idempotent semiring comes with a partial order defined by $a\leq b\iff a+b=b$. Both addition and multiplication preserve this partial order. One consequence of this is the following well-known fact.

\begin{lemma}\label{lem:zero-sum-free}
    Let $A$ be an additively idempotent semiring, then it is zero-sum free, i.e., if $a, b \in A$ with $a+b = 0$, then $a = b = 0$.
\end{lemma}

\begin{proof}
    If $a+b = 0$, then $a+a+b = a+0$ and so $a+b = a$ implying that $b \leq a$ and by symmetry $a \leq b$. So $a=b$ and $a+ b = a+a = a = 0$.
\end{proof}

We will denote by $\mathbb{B}$ the semifield with two elements $\{1,0\}$, where $1$ is the multiplicative identity, $0$ is the additive identity and $1+1 = 1$. 
The {\it tropical semifield}, denoted $\mathbb{T}$, is the set $\mathbb{R}  \cup \{-\infty\} $ with the $+$ operation defined to be the maximum and the $\cdot$ operation defined to be the usual addition, with $-\infty = 0_\mathbb{T}$. 

In order to distinguish when we are considering a real number as being in $\R$ or being in $\T$ we introduce some notation. For a real number $a$, we let $t^a$ denote the corresponding element of $\T$. In the same vein, given $\mathfrak{a}\in\T$, we write $\log(\mathfrak{a})$ for the corresponding element of $\R\cup\{-\infty\}$. This notation is motivated as follows.
Given a non-archimedean valuation $\nu: K \rightarrow \R\cup\{-\infty\}$ on a field and $\lambda \in \R$ with $\lambda>1$, we get a non-archimedean absolute value $|\cdot|_\nu:K\rightarrow[0, \infty)$ by setting $|x|_{\nu}=\lambda^{\nu(x)}$. Since $\T$ is isomorphic to the semifield $\big([0,\infty),\max,\cdot_{\R}\big)$, we use a notation for the correspondence between elements of $\R\cup\{-\infty\}$ and elements of $\T$ that is analogous to the notation for the correspondence between $\nu(x)$ and $|x|_{\nu}$. This notation is also convenient as we get many familiar identities such as $\log(1_{\T})=0_{\R}$ and $t^a t^b=t^{a+_{\R}b}$.

\subsection{Congruences}

\begin{defi}
Let $A$ be a semiring and let $a \in A\sdrop\{0\}$. We say that $a$ is a \emph{cancellative element} if, for all $b, c \in A$, whenever $ab=ac$ then $b=c$. If all elements of $A\sdrop\{0\}$ are cancellative, then we say that $A$ is a \textit{cancellative semiring}.
\end{defi}

\begin{definition}[cf.\ Definition 2.3 and Proposition 2.10 in \cite{JM17}]\label{def: prime_domain}
A \textit{congruence} on a semiring $A$ is an equivalence relation on $A$ that respects the operations of $A$. The \emph{trivial congruence} on $A$ is the diagonal $\Delta\subseteq A\times A$, for which $A/\Delta\cong A$. We call a congruence $C$ on $A$ \emph{cancellative} if $A/C$ is a cancellative semiring. We call a proper congruence $P$ of an idempotent semiring $A$  {\it prime} if $A/P$ is totally ordered and cancellative. %\pinky{(cf.\ Definition 2.3 and Proposition 2.10 in \cite{JM17})}. %(cf. Definition~\ref{def: prime2}) %(cf.\ Definition 2.3 and Proposition 2.10 in \cite{JM17}).
\end{definition}

Let $C$ be a congruence on a semiring $A$. The \emph{ideal-kernel} of $C$ is the set of elements $a \in A$ such that $(a,0_A) \in C$. It is easy to see that the ideal-kernel of $C$ is an ideal of $A$. If the ideal-kernel of $C$ is the zero ideal, then we say that $C$ has \emph{trivial ideal-kernel}.

Let $\ph:A\to A'$ be a homomorphism of semirings. The \emph{congruence-kernel} of $\ph$, denoted $\ker(\ph)$, is the pullback $\ph^*(\Delta)$ of the trivial congruence on $A'$. 
That is, $\ker(\ph):=\{(f,g)\in A\times A\,:\,\ph(f)=\ph(g)\}$.
If $A'$ is totally ordered and cancellative, then $\ker(\ph)$ is prime. The \emph{ideal-kernel} of $\ph$ is the set of those $a\in A$ such that $\ph(a)=0_{A'}$. Clearly, the ideal-kernel of $\ph$ is the same as the ideal-kernel of $\ker(\ph)$. 

\begin{notation}\label{notation:leqP}
Let $R$ be semiring and let $P$ be a prime congruence on $R$. For two elements $r_1, r_2 \in R$ we say that $r_1 \leq_P r_2$ (resp. $r_1 <_P r_2$, resp. $r_1 \equiv_P r_2$) whenever $\overline{r_1} \leq \overline{r_2}$ (resp. $\overline{r_1} <  \overline{r_2}$, resp. $\overline{r_1} = \overline{r_2}$) in $R/P$. Here by $\overline{r}$ we mean the image of $r$ in $R/P$.
\end{notation}

Note that $P$ is determined by the relation $\leq_P$ on $R$. Moreover, if $\mathcal{G}$ is a set of additive generators for $R$, then $P$ is determined by the restriction of $\leq_P$ to $\mathcal{G}$. 

\begin{defi}\label{def:ResidueSemifield}
Let $P$ be a prime congruence on a semiring $A$. The {\it residue semifield of $A$ at $P$}, denoted $\kappa(P)$, is the total semiring of fractions\footnote{This is obtained via the usual construction as equivalence classes of pairs. For more details see Definition 2.15 in \cite{FM22}} of $A/P$. Since $A/P$ is cancellative, $\kappa(P)$ is, indeed, a semifield. We denote the canonical homomorphism $A\to A/P\to\kappa(P)$ by $a\mapsto|a|_P$.
\end{defi}

When $\base\subseteq \TT$ we describe the prime congruences on $\base[\ZZ^n]$ in terms of their \emph{defining matrices}. Every prime congruence on $\base[\Z^n]$ has a defining matrix; see \cite{JM17}, where it is shown that the matrix can be taken to be particularly nice. 

A $k\times (n+1)$ real valued matrix $\Theta$ such that the first column of $\Theta$ is lexicographically greater than or equal to the zero vector gives a prime congruence on $\base[\Z^n]$ as follows:
For any monomial $m=t^a\chi^u\in\base[\Z^n]$ we let $\Phi(m)=\Theta\begin{pmatrix}a \\ u\end{pmatrix}\in\R^k$, where we view $u\in\Z^n$ as a column vector. We call $\begin{pmatrix}a \\ u\end{pmatrix}$ the \emph{exponent vector} of the monomial $m$. For any nonzero $f\in\base[\Z^n]$, write $f$ as a sum of monomials $m_1,\ldots,m_r$ and set $\Phi(f)=\max_{1\leq i\leq r}\Phi(m_i)$, where the maximum is taken with respect to the lexicographic order. Finally, set $\Phi(0)=-\infty$. We specify the prime congruence $P$ by saying that $f$ and $g$ are equal modulo $P$ if $\Phi(f)=\Phi(g)$. In this case we say that $\Theta$ is a defining matrix for $P$. The first column of $\Theta$ is called the column corresponding to the coefficient or the column corresponding to $\base$. For $1\leq i\leq n$, we say that the $(i+1)^{\text{st}}$ column of $\Theta$ is the column corresponding to $x_i$.

\begin{example}
Let $\base [\Z^2] = \T[x_1^{\pm1},x_2^{\pm1}]$. The matrix $\begin{pmatrix} 0& 1 & 0\\ 0& 0 & 1 \end{pmatrix}$ represents the lexicographic order on this semiring with $x_1 \gg x_2$. The corresponding prime congruence contains $(f,g)$ if $f$ and $g$ have the same initial term with respect to this order.\exEnd
\end{example}

In contrast, the congruences on $\base[\N^n]$ are allowed to have non-trivial ideal-kernel. We account for this by allowing $-\infty$ as an entry of the matrix, so long as the whole column consists of $-\infty$'s as seen in the next example. Here we use the convention that $0\cdot(-\infty)=0$.

\begin{example}
     Let $\base [\N^n] = \T[x_1, x_2]$. The congruence $C = \left< (x_1, 1), (x_2, 0)\right>$
    is given by the matrix $\begin{pmatrix} 1& 0 & -\infty\end{pmatrix}.$
    When a congruence has non-trivial ideal-kernel, we note that it has the same quotient as a congruence with trivial ideal-kernel on a smaller semiring. In this example, if we let $C'=\left< (x_1, 1)\right>$ on $\T[x_1]$ then $\T[x_1, x_2]/C = \T[x_1]/C'$. Geometrically, the rows of a matrix with infinities in the same columns are in the same stratum of a toric variety.  
    \exEnd
\end{example}

\begin{definition}[adapted from Definition 5.1.1 in \cite{GG13}]\label{def:bend_rel}

Let $A$ be an idempotent semiring, let $\Mon$ be a monoid, and let $f \in A[\Mon]$. We let $\supp(f)$ denote the support of $f$ and, for $i$ in $\supp(f)$, we write $f_{\hat \imath}$ for the result of deleting the $i$ term from $f$. Then the set of \emph{bend relations of $f$} is the set of pairs 
$$\bend(f) = \{(f , f_{\hat \imath})\}_{i\in \supp(f)}.$$
 
For any set $I\subseteq A[\Mon]$, the \emph{bend congruence} of $I$, denoted $\Bend I$,  is the congruence generated by the bend relations of all $f\in I$. 
\end{definition}

\begin{definition}
    Let $C$ be a congruence. The \textit{radical of $C$}, denoted $\sqrt{C}$, is the intersection of all prime congruences containing $C$. If $C = \sqrt{C}$, we say that that $C$ is a \emph{radical} congruence. 
\end{definition}

\begin{proposition}[Proposition 3.10 \cite{JM17}]
Let $A$ be an idempotent semiring. Then every cancellative congruence on $A$ is radical. 
\end{proposition}

\begin{coro}\label{coro:radical_by_generalized_power}
For any congruence $C$ of an additively idempotent semiring $A$, we have that
$$\sqrt{C}=\left\{ (x,y)\in A\times A \,:\, \text{there are }k\in\Z_{\geq0}\text{ and }c\in A \text{ such that } \left((x+y)^k+c\right)(x,y)\in C \right\}$$
\end{coro}
\begin{proof}
This follows from Theorem 3.9 and Proposition 5.2 in \cite{JM17}.
\end{proof}

\subsection{
Polyhedral geometry and basics of tropical toric varieties}
In this section we recall the definition of classical tropical toric varieties.

\par\medskip  
For a general introduction to toric varieties see \cite{CLS} or \cite{Ful93}. 
In this paper we follow conventions from \cite{Rab10}.
Let $\Lattice$ be a finitely generated free abelian group and denote by $\Lattice_\R $ the vector space $ \Lattice\otimes_{\Z}\R$. We write $N:=\Lattice^*=\Hom(\Lattice,\Z)$ and $N_\R=\Hom(\Lattice,\R)\cong N\otimes_{\Z}\R$. We will use the pairing $N_\R\times\Lattice_{\R}\to\R$ given by $(v,u)\mapsto\angbra{v,u}:=v(u)$. 
A cone $\sigma$ is a \textit{strongly convex rational polyhedral cone} in $N_\R$ if $\sigma = \sum\limits_{i=1}^{r}\R_{\geq 0}v_i$ for $v_i \in N$ and $\sigma$ contains no line. %By a \textit{cone} we will always mean a strongly convex rational polyhedral cone. 
We write $\tau\leq\sigma$ to mean that $\tau$ is a face of $\sigma$. 
We denote by $\sigma^\vee$ the dual cone of $\sigma$,
$$\sigma^\vee = \{u\in \Lattice_\R : \left< v,u \right> \leq 0, \forall v\in \sigma\} = \bigcap\limits_{i = 1}^r \{u\in \Lattice_\R : \left< v_i,u \right> \leq 0\}.$$
A \emph{fan} is a collection $\Sigma$ of strongly convex rational polyhedral cones in $N_{\R}$ such that (1) if $\tau\leq\sigma\in\Sigma$ then $\tau\in\Sigma$ and (2) if $\sigma_1,\sigma_2\in\Sigma$ then $\sigma_1\cap\sigma_2$ is a face of both $\sigma_1$ and $\sigma_2$.

\begin{definition}
    Let $P$ be a polyhedron in $N_{\R}$, and pick a presentation $P = \{ x\in N_\R: \left< a_i, x\right> \leq b_i\text{ for }i=1,\ldots,q\}$, for given vectors $a_1,\ldots,a_q\in \Lattice_{\R}$ and $b_1,\ldots,b_q \in \R$. 
    The \emph{recession cone} of $P$ is 
    %$$\text{rec}(P) = \{x\in N_{\R} \,:\, \text{for all }y\in P,\ x+y\in P\} = \{ x\in N_\R: \left< a_i, x\right> \leq 0 \text{ for }i=1,\ldots,q\}.$$ 
\begin{align*}
    \rec(P) &= \{x\in N_{\R} \,:\, \text{for all }y\in P,\ x+y\in P\} = \{ x\in N_\R: \left< a_i, x\right> \leq 0 \text{ for }i=1,\ldots,q\}\\
    &=\{x\in N_{\R} \,:\, \text{there is a }y\in P\text{ such that for all }r\in\R_{\geq0}, y+rx\in P\}.
\end{align*} 
    If $P=\emptyset$, we let $\rec(P)=\{0_{N_{\R}}\}$.
    If $\Sigma$ is a polyhedral complex in $\R^n$, then its \emph{recession}\footnote{We do not refer to the \emph{recession fan} as the set of recession cones is not always a fan.} $\rec(\Sigma)$ is the union of all cones $\rec(P)$ where $P$ runs over $\Sigma$. Note that 
    $$\rec(\Sigma)=\{x\in N_{\R} \,:\, \text{there is a }y\in |\Sigma|\text{ such that for all }r\in\R_{\geq0}, y+rx\in |\Sigma|\},$$
    where $|\Sigma|=\dcup_{P\in\Sigma}P$ is the \emph{support} of $\Sigma$. We therefore extend both of these definitions by saying that if $\frakP$ is a polyhedral set (that is, a finite union of polyhedra), then 
    $$\rec(\frakP)=\{x\in N_{\R} \,:\, \text{there is a }y\in \frakP\text{ such that for all }r\in\R_{\geq0}, y+rx\in \frakP\}.$$
\end{definition}

It follows from this last definition that, if $\frakP$ and $\frakQ$ are polyhedral sets, then $\rec(\frakP\cup\frakQ)=\rec(\frakP)\cup\rec(\frakQ)$ and that $\frakP\subseteq\frakQ\implies\rec(\frakP)\subseteq\rec(\frakQ)$. On the other hand, in general we only have $\rec(\frakP\cap\frakQ)\subseteq\rec(\frakP)\cap\rec(\frakQ)$.

\begin{definition}
    We call a monoid of the form $\Mon = \sigma^\vee \cap \Lattice$ a \textit{toric monoid}. When we want to specify $\sigma$, we say that $\Mon$ is the \textit{toric monoid corresponding to} $\sigma$.
\end{definition}

Because $\sigma$ is strongly convex, we can identify $\Lattice$ with the groupification of $\Mon$.

We let $\nbar{\R}=\R\cup\{-\infty\}$ considered as a monoid with the operation $+_{\R}$.

\begin{definition} 
Let $\sigma$ be a strongly convex rational polyhedral
cone in $N_\R$. We denote by $N_{\R}(\sigma)$ the set of monoid maps $\Hom(\sigma^\vee \cap\Lattice, \nbar{\R})$. We call $N_{\R}(\sigma)$ the \emph{tropical toric variety} corresponding to the cone $\sigma$.
\end{definition}

The tropical toric variety $N_{\R}(\sigma)$ is endowed with a topology, given as follows. First, observe that the map $(\nbar{\R}, +_{\R}) \to (\R_{\geq 0}, \cdot_\R)$ given by $x\mapsto e^x$ is an isomorphism of monoids. Thus, we can use the usual topology on $\R_{\geq0}$ to make $\nbar{\R}$ into a topological monoid. So we can give 
$N_{\R}(\sigma) = \Hom(\sigma^\vee\cap\Lattice,\nbar{\R}) \subseteq \left(\nbar{\R}\right)^{\sigma^\vee\cap\Lattice}$
the subspace topology where the space $\left(\nbar{\R}\right)^{\sigma^\vee\cap\Lattice}$ of all functions $\sigma^\vee\cap\Lattice\to\nbar{\R}$ carries the product topology. This makes $N_{\R}(\sigma)$ into a topological monoid using the operation $+_{\R}$. 
For any face $\tau$ of $\sigma$, the inclusion $\sigma^\vee\cap\Lattice\subseteq\tau^\vee\cap\Lattice$ induces a map $N_{\R}(\tau)\to N_{\R}(\sigma)$, which identifies $N_{\R}(\tau)$ with an open topological submonoid of $N_{\R}(\sigma)$.
%The copy of $N_{\R}\cong N_{\R}(\{0\})$ in $N_{\R}(\sigma)$ is exactly the set of units for the operation $+_{\R}$.
If $\Sigma$ is a fan then we can glue the $N_{\R}(\sigma)$s for $\sigma\in\Sigma$ along these inclusion maps to get the \emph{tropical toric variety $N_{\R}(\Sigma)$ corresponding to $\Sigma$.}

For any strongly convex rational polyhedral cone $\sigma$ there is a bijection 
$$\bigsqcup_{\tau\leq\sigma}N_{\R}/\vspan{\tau}\to N_{\R}(\sigma)$$
given as follows. We send $w\in N_{\R}/\vspan{\tau}$ to the function $\ph_w:\sigma^\vee\cap\Lattice\to\nbar{\R}$ given by $\ph_w(u)=\begin{cases}\angbra{u,w}&\text{if }u\in\tau^\perp\\-\infty&\text{if }u\notin\tau^\perp\end{cases}$; we will write $\angbra{u,w}$ for $\ph_w(u)$ even when this value is $-\infty$. We let $N_{\R}/\tau$ denote the copy of $N_{\R}/\vspan(\tau)$ in $N_{\R}(\sigma)$. Then each $N_{\R}/\tau$ is a monoid that is a sub-semigroup (but not sub-monoid) of $N_{\R}(\sigma)$ (the units of the two monoids are not the same). The monoid structure and topology on $N_{\R}/\tau$ are the same as the addition and topology on $N_{\R}/\vspan(\tau)$ as a quotient vector space.
The copy of $N_{\R}\cong N_{\R}/\{0\}$ in $N_{\R}(\sigma)$ is exactly the set of units for the operation $+_{\R}$ on $N_{\R}(\sigma)$.

We call each $N_{\R}/\tau$ a \emph{stratum} of the tropical toric variety.

\begin{remark}
We can now extend the idea of a matrix defining a prime congruence on $\base[\Z^n]$ or $\base[\N^n]$ to prime congruences on $\base[\Mon]$. To do this, the rows of the matrix are now elements of $\R_{\geq0}\times N_{\R}(\sigma)$. We also require that all of the rows of the matrix come from $\R_{\geq0}\times(N_{\R}/\tau)$ for a single $\tau\leq\sigma$; this corresponds to the condition that in a matrix for a prime of $\base[\N^n]$, if a $-\infty$ appears then the whole column must be $-\infty$'s.
\end{remark}

\subsection{Tropical varieties} 

Let $A = \base[\Mon]$, where $\base\subseteq \T$ and $\Mon$ is a toric monoid corresponding to the cone $\sigma$.

Let $I$ be a set of polynomials of $A$ and let $C\subseteq A\times A$ be a set of pairs. In this paper we will consider the following varieties:

\begin{equation*}
\begin{split}
% & = \{P \text{ geometric prime congruence} : \text{ for all } f \in I, \bend(f) \subseteq P\},  \\
V(C) &= \{ a \in N_{\R}(\sigma) : f(a) = g(a), \text{ for all } (f,g) \in C\}\\ 
%&= \{P \text{ geometric prime congruence } : P \supseteq C\}
V(I) &= \{a \in N_{\R}(\sigma) : \forall f \in I, f \text{ tropically vanishes at } a\} \\
&= V(\Bend(I)).
\end{split}
\end{equation*}

Let $I$ be an ideal of $\base[\Mon]$, and let $\{f_1, \ldots, f_k\} \subseteq I$ such that 
%$$V(I) = V(f_1, \ldots, f_k) = \bigcap_{i=1}^k V(f_i),$$ 
$V(I) = \bigcap_{i=1}^k V(f_i),$ 
then we say that $\{f_1, \ldots, f_k\}$ is a \emph{tropical basis} for $I$. If $k < \infty$, we say that $I$ has a \emph{finite tropical basis}.

We refrain from making an analogous definition for congruences until Section 3, where we will define an enhanced version of $V(C)$. 

\newcommand{\theFiniteSet}{D}
\begin{definition}[Adapted from Definition 1.1 of \cite{MR18}]\label{def: trop_ideals}
Let $\base$ be a sub-semifield of $\T$, let $\Mon$ be any monoid, and let $I\subseteq\base[\Mon]$ be an ideal. For any $\theFiniteSet\subseteq \Mon$ we let $I_\theFiniteSet$ denote the set of those $f\in I$ whose support is contained in $\theFiniteSet$.

We say that $I$ is a \emph{tropical ideal} if, for every finite set $\theFiniteSet\subset\Mon$, $I_{\theFiniteSet}$ is a tropical linear space, i.e., $I_{\theFiniteSet}$ is the set of vectors of a valuated matroid. 
To be very concrete, $I$ is a tropical ideal if it satisfies the following monomial elimination axiom:

For any $f, g \in I$ and any monomial ${\pmb x}^{\pmb u}$ which appears in both $f$ and $g$ with the same nonzero coefficient, there exists a polynomial $h \in I$ such that $\pmb u$ is not in the support of $h$ and, for any other $\pmb v\in \Mon$, the coefficient $c_{\pmb v}(h)$ of ${\pmb x}^{\pmb v}$ in $h$ satisfies $c_{\pmb v}(h) \leq \max (c_{\pmb v}(f), c_{\pmb v}(g))$ with equality holding whenever $c_{\pmb v}(f) \neq c_{\pmb v}(g)$. Note that, in this case, we have $\supp(h)\subseteq \Big(\supp(f)\cup\supp(g)\Big)\sdrop\{\pmb u\}$, so if $f,g\in I_\theFiniteSet$ then $h\in I_\theFiniteSet$.
\end{definition}

If $I$ is a tropical ideal, then by \cite[Theorem 5.9]{MR18} $I$ has a finite tropical basis.

\begin{example}
Let $\base$ be a sub-semifield of $\T$, let $\Mon$ be any monoid, and let $k\onto\base$ be a surjective valuation. Then, for any ideal $I$ of $k[\Mon]$, $\trop(I):=\{\trop(f)\,:\, f\in I\}$ is a \emph{realizable} tropical ideal in $\base[\Mon]$. 
\exEnd
\end{example}

We now give a few basic lemmas that will be useful later for constructing examples.
While these lemmas are known to experts, they have not been proven in the literature, so we provide proofs here.

\begin{lemma}
For any $f,g\in\base[\Mon]$, $V(fg)= V(f)\cup V(g)$.
\end{lemma}

\begin{proof}
We first show $V(f)\cup V(g)\subseteq V(fg)$. By symmetry, it suffices to show $V(f)\subseteq V(fg)$. If $f$ is constant, we get $V(f)\subseteq V(fg)$ by considering cases as to whether $f=0$ or not. If $f$ is a non-constant monomial and $w\in V(f)$ then $f(w)=-\infty$, so $(fg)(w)=-\infty$, so $w\in V(fg)$.

Now suppose $f$ has multiple terms. If $w\in V(f)$, then there are terms $m_1,m_2$ of $f$ such that $f(w)=m_1(w)=m_2(w)$. Let $\mu$ be any term of $g$ for which $g(w)=\mu(w)$. Then $(fg)(w)=(\mu m_1)(w)=(\mu m_2)(w)$, with $\mu m_1, \mu m_2$ terms of $fg$, so $w\in V(fg)$.

For the other direction, suppose that $w\in N_{\R}(\sigma)$ is not in $V(f)\cup V(g)$; we want to show that $w$ is not in $V(fg)$. We cannot have $f=0$ or $g=0$, for then $V(f)=N_{\R}(\sigma)$ or $V(g)=N_{\R}(\sigma)$. If $f$ (or $g$) is a nonzero constant then $V(f)=\emptyset$ (resp.\ $V(g)=\emptyset$) and $V(fg)=V(g)$ (resp.\ $V(fg)=V(f)$), so we are done. If $f$ and $g$ are both monomials then $f(w)\neq-\infty$ and $g(w)\neq-\infty$ so $(fg)(w)\neq-\infty$, giving us $w\notin V(fg)$.

So now suppose that $f,g$ are both nonzero and that at least one of $f,g$ has multiple terms; we assume without loss of generality that $f$ has multiple terms. In this case $fg$ has multiple terms, so we want to show that the maximum in $fg(w)$ is attained only once. Since $w\notin V(f)$, there is a term $m$ of $f$ such that, for all other terms $m'$ of $f$, $m'(w)<m(w)$. In particular, $m(w)\neq-\infty$. Since $w\notin V(g)$, there is a term $\mu$ of $g$ such that $g(w)=\mu(w)$, $\mu'(w)<\mu(w)$ for all other terms $\mu'$ of $g$ (if there are any other terms), and $\mu(w)\neq-\infty$ (whether or not there are any other terms).

Any term of $fg$ other than $m\mu$ can be written as $m'\mu'$ where either $m'\neq m$ or $\mu'\neq \mu$. If $m'\neq m$ then $m'(w)< m(w)$, so $m'\mu'(w)\leq m'\mu(w)<m\mu(w)$. Similarly, if $\mu'\neq\mu$ then $\mu'(w)<\mu(w)$ so $m'\mu'(w)\leq m\mu'(w)<m\mu(w)$. Thus any term of $fg$ other than $m\mu$ evaluates to less than $m\mu(w)$. So $w\notin V(fg)$.
\end{proof}

\begin{lemma}\label{lemma:V(Polynomial)ContainsV(Generators)}
Let $f_1,\ldots,f_k,a_1,\ldots,a_k\in\base[\Mon]$. Then $V(a_1f_1+\cdots+a_kf_k)\supseteq V(f_1)\cap\cdots\cap V(f_k)$.
\end{lemma}
\begin{proof}
An elementary induction argument shows that it suffices to show the result for $k=2$. So suppose $a,b,f,g\in\base[\Mon]$; we want to show that $V(af+bg)\supseteq V(f)\cap V(g)$.

The previous lemma shows that $V(af)\supseteq V(f)$ and $V(bg)\supseteq V(g)$, so it suffices to show that $V(af+bg)\supseteq V(af)\cap V(bg)$. By renaming $af$ and $bg$, it suffices to show that $V(f+g)\supseteq V(f)\cap V(g)$.

If $f$ or $g$ is zero, this is trivial, so suppose both $f$ and $g$ are nonzero. We now consider the case where $f$ and $g$ are monomials. If $f+g$ is a monomial then it is one of $f$ or $g$ and we are done. So suppose $f+g$ is a binomial. Any $w\in V(f)\cap V(g)$ then satisfies $f(w)=-\infty=g(w)$, i.e., $w$ satisfies the unique bend relation $f\sim g$ of $f+g$. Thus $w\in V(f+g)$.

Now suppose that one of $f$ and $g$ is a monomial and the other is not. We assume without loss of generality that $f$ has multiple nonzero terms and $g$ is a monomial. If $w\in V(f)\cap V(g)$ then there are terms $m_1,m_2$ of $f$ such that $f(w)=m_1(w)=m_2(w)$ and $g(w)=-\infty$. So $(f+g)(w)=\max(f(w),-\infty)=f(w)=m_1(w)=m_2(w)$. A simple case analysis as to whether $m_1$ and $m_2$ are terms of $f+g$ (with the same coefficients) then shows that $w\in V(f+g)$.

Now suppose that both $f$ and $g$ have multiple terms. Say $w\in V(f)\cap V(g)$. So there are terms $m_1,m_2$ of $f$ such that $f(w)=m_1(w)=m_2(w)$ and terms $\mu_1,\mu_2$ of $g$ such that $g(w)=\mu_1(w)=\mu_2(w)$. We may assume without loss of generality that $f(w)\geq g(w)$. Then $m_1$ and $m_2$ are terms of $f+g$ and $(f+g)(w)=m_1(w)=m_2(w)$, so $w\in V(f+g)$.
%The arguement that m_1 and m_2 are terms of the sum comes from the fact that $m_1(w)=f(w)\geq g(w)\geq \mu'(w)$ for any term $\mu'$ of $g$, so if $m_1$ and $\mu'$ have the same exponent vector, then the coefficient of $\mu'$ must be $\leq$ the coefficient of $m_1$. 
\end{proof}

\begin{lemma}\label{lemma:TropBasisInGeneratingSet}
Let $I\subseteq\base[\Mon]$ be an ideal generated by a set $G\subseteq\base[\Mon]$. Then $V(I)=V(G)$. Moreover, if $I$ has a finite tropical basis then $I$ has a finite tropical basis contained in $G$.
\end{lemma}\begin{proof}
For any set $X\subseteq I$, $V(X)\supseteq V(I)$. In particular, $X$ is a tropical basis if and only if $V(X)\subseteq V(I)$.

Since $V(I)=\dcap_{f\in I}V(f)$, to prove that $V(I)=V(G)$ it suffices to show that, for all $f\in I$, $V(f)\supseteq V(G)$. Fix $f\in I$.

Since $f$ is in the ideal $I$ generated by $G$, there are $g_1,\ldots,g_k\in G$ and $a_1,\ldots,a_k\in \base[\Mon]$ such that $f=a_1g_1+\cdots+a_kg_k$. So by Lemma~\ref{lemma:V(Polynomial)ContainsV(Generators)} 
%$V(f)=V(a_1g_1+\cdots+a_kg_k)\supseteq V(g_1)\cap\cdots\cap V(g_k)\supseteq V(G)$.
$$V(f)=V(a_1g_1+\cdots+a_kg_k) \supseteq V(g_1)\cap\cdots\cap V(g_k) \supseteq V(G).$$

Now suppose that $I$ has a finite tropical basis $\{f_1,\ldots,f_m\}$. For each $i=1,\ldots,m$ write $f_i=\dsum_{j=1}^{k_i}a_{i,j}g_{i,j}$ with $g_{i,j}\in G$ and $a_{i,j}\in \base[\Mon]$. Let $X=\{g_{i,j}\,:\,1\leq i\leq m,\ 1\leq j\leq k_i\}$. Then
\begin{align*}
V(I)&=V(f_1,\ldots,f_m)=\dcap_{i=1}^m V(f_i)=\dcap_{i=1}^m V\left(\dsum_{j=1}^{k_i}a_{i,j}g_{i,j}\right)\supseteq\dcap_{i=1}^m \dcap_{j=1}^{k_i}V\left(g_{i,j}\right)=V(X)
\end{align*}
by Lemma~\ref{lemma:V(Polynomial)ContainsV(Generators)}, so $X$ is a finite tropical basis for $I$ that is contained in $G$.
\end{proof}

\subsection{Prime filters and flags of polyhedra}
We recall some needed definitions and results from our previous work in \cite{FM23}. 
%While in \cite{FM23}  we focused on the case of the torus, we also need the remaining strata of a (tropical) toric variety. 
While in \cite{FM23}  we focused on the case of the torus, in this paper we work with all of the strata of a tropical toric variety. 
Below we adapt the relevant definitions and results from \cite{FM23} to include the boundary strata. 

Let $\sigma$ be a cone in $N_{\R}$. For $\tau\leq\sigma$, we let $K_\tau$ be the set of those terms $a\chi^u$ in $\base[\Mon]$ such that $u\notin\tau^\perp$. Given any prime congruence $P$ on $\base[\Mon]$, there is a face $\tau\leq\sigma$ such that the intersection of the ideal-kernel of $P$ with the set $\Terms{\base[\Mon]}$ of terms is $K_\tau$. We let $V_\tau$ denote the set of such prime congruences $P$. If $P\in V_\tau$ we may also write $\tau=\tau_P$.

\begin{notation}
    Throughout the paper we will denote by $\Gamma:=\log(\base^\times)$ the subgroup of $\R$ corresponding to $\base$. 
\end{notation}

\begin{definition}\label{def:flagOfCones}
We let $\calC_\bullet=(\calc_{-1}\leq\calC_0\leq\calC_1\leq\cdots\leq\calC_{k})$ denote a \emph{flag of cones} in $\R_{\geq_0}\times N_{\R}/\tau$ with $\dim\calC_{i}=i+1$. We say that $\calC_\bullet$ is \emph{simplicial} if each $\calC_i$ is a simplicial cone.
\end{definition}
\noindent We often leave off $\calc_{-1}$ from the notation because it is always the zero cone.

A cone in $\R_{\geq0}\times N_{\R}/\tau$ is \emph{$\Gamma$-admissible} if it can be written as 
$$\{(r,x)\in \R_{\geq0}\times N_{\R}/\tau\ :\ r\gamma_i+\angbra{x,u_i}\leq0\text{ for }i=1,\ldots,q\}$$
for some $u_1,\ldots,u_q\in \Lattice\cap\tau^\perp$ and $\gamma_1,\ldots,\gamma_q\in\Gamma$. Note that we can also allow $u_i\in\Mon\sdrop\tau^\perp$, for then $r\gamma_i+\angbra{x,u_i}=r\gamma_i+(-\infty)=-\infty\leq0$.
%A polyhedron $\calQ$ is $\Gamma$-rational if and only if $c(\calQ)$ is $\Gamma$-admissible. 
A cone contained in $\{0\}\times N_{\R}/\tau$ is $\Gamma$-admissible if and only if it is rational. We call a finite union of $\Gamma$-admissible cones a \emph{$\Gamma$-admissible \conoidSet}. Note that the set of $\Gamma$-admissible \conoidSets\ forms a lattice when ordered by inclusion; meet and join are given by intersection and union, respectively.

\begin{definition}\label{def: gamma-rational-nbhd}

Let $\calc_\bullet$ be a flag of cones. A \emph{$\Gamma$-admissible neighborhood} of $\calc_\bullet$ is a $\Gamma$-admissible cone $\calD$ which meets the relative interior of each $\calc_i$.
\end{definition}

\begin{defi}\label{def:PrimeOfAFlag}%Definition 3.3
Let $\calc_\bullet$ be a simplicial flag of cones. The \emph{prime congruence $P_{\calc_\bullet}$ defined by $\calC_\bullet$} is the prime congruence %\green{whose defining matrix has rows $w_0,\ldots, w_k$ with $ w_i\in\calc_i\sdrop\calc_{i-1}$.}   
given by any matrix that has rows $w_0,\ldots, w_k$ with $ w_i\in\calc_i\sdrop\calc_{i-1}$. This prime congruence is independent of the choice of $w_0,\ldots,w_k$.

If $\calc_\bullet$ is the flag whose only cone is the ray generated by some nonzero $w\in N_{\R}(\sigma)$ then we also write $P_{w}$ for $P_{\calc_\bullet}$.
\end{defi}

Given $f=\dsum_{u\in\Mon}f_u\chi^u\in \base[\Mon]$ and $(r,x)\in \R_{\geq0}\times N_{\R}(\sigma)$ we let $$\wt{f}(r,x)=\displaystyle\max_{u\in\Mon}\left(r\log(f_u)+_{\R}\angbra{x,u}\right),$$
which can be viewed as a homogenization of the usual function that $f$ defines on $N_{\R}(\sigma)$. For any $f_0,f_1,\ldots,f_n\in\base[\Mon]$ and $\tau\leq\sigma$, we define 
$$\wt{R}_{\tau}(f_0,f_1,\ldots,f_n):=\{ w\in\R_{\geq0}\times(N_{\R}/\tau)\;:\;\wt{f}_0(w)\geq\wt{f}_i(w)\text{ for }1\leq i\leq n\}.$$
%By $N_{\R}/\tau$ we denote the stratum of $N_{\R}(\sigma)$ corresponding to $\tau$. 
Note that, if $m_0,m_1,\ldots,m_n$ are monomials, then $\wt{R}_\tau(m_0,m_1\ldots,m_n)$ is a polyhedral cone in $\R_{\geq0}\times (N_{\R}/\tau)$.

Recall that, for a prime congruence $P$ on a semiring $A$, the canonical homomorphism $A\to A/P\to\kappa(P)$ is denoted by $|\bullet|_P$.

\begin{defi}
For any prime congruence $P$ on $\base[\Mon]$, the \emph{prime filter $\wt{\calf}_P$ that $P$ defines on the \ConoidLarry}\ is the collection of $\Gamma$-admissible \conoidSets\ $\wt{U}$ for which there are $f_0,f_1, \ldots, f_n \in \base[\Mon]$ such that $\wt{R}_\tau(f_0,f_1, \ldots, f_n ) \subseteq \wt{U}$ and $|f_0|_P \geq |f_i|_P$ for $1\leq i\leq n$. Here $\tau=\tau_P$. 

\end{defi}
\noindent

In particular, if $|f|_P\geq|g|_P$ then $\wt{R}_\tau(f,g)\in\wtcalf_P$.

\begin{theorem}\label{thm:TheFilterHasFlagMeaning}
For any prime $P$ on $\base[\Mon]$, let $\wtcalf=\wtcalf_P$ and pick a simplicial flag $\calc_\bullet$ of cones such that $P=P_{\calc_\bullet}$. Then, for any $\Gamma$-admissible cone $\wt{U}$ in $\R_{\geq0}\times N_{\R}/\tau_P$, $\wt{U}\in\wtcalf$ if and only if $\wt{U}$ is a neighborhood of $\calc_\bullet$.
\end{theorem}
\begin{proof}
    This is Theorem 3.20 in \cite{FM23} but stated for the toric stratum $N_{\R}/\tau_P$ instead of $N_{\R}$.
\end{proof}

\section{Congruences and flags of cones}\label{sec:Filters}

Let $\base$ be a sub-semifield of $\T$, let $\Gamma$ be the corresponding additive subgroup of $\R$, and let $\Mon$ be a toric monoid corresponding to a (strictly convex rational polyhedral) cone $\sigma\subseteq N_{\R}$. 
Recall, that given $f=\dsum_{u\in\Mon}f_u\chi^u\in \base[\Mon]$ and $(r,x)\in \R_{\geq0}\times N_{\R}(\sigma)$, we let $\wt{f}(r,x)=\displaystyle\max_{u\in\Mon}\left(r\log(f_u)+_{\R}\angbra{x,u}\right)$. For any $E\subseteq (\base[\Mon])^2$ we set 
$$\wt{V}(E):=\{w\in\R_{\geq0}\times N_{\R}(\sigma)\;:\;\wt{f}(w)=\wt{g}(w)\text{ for all }(f,g)\in E\};$$
$E$ will usually be a congruence on $\base[\Mon]$.

\begin{defi}\label{def: finiteTropBasis}
Let $E$ be a congruence on $\base[\Mon]$. We say that $E$ \emph{has a finite tropical basis} if there are $(f_1,g_1),\ldots,(f_r,g_r)\in E$ such that $\wt{V}(E)=\wt{V}\big((f_1,g_1),\ldots,(f_r,g_r)\big)$.
\end{defi}

\begin{example}
    For any tropical ideal $I$ in $\T[\Mon]$, $E=\Bend(I)$ has finite tropical basis. Towards seeing this, we let $\calb_1$ be a finite tropical basis for $I$.

    Consider the map $\varphi : \T[\Mon] \to \B[\Mon]$. The ideal $\varphi(I)$ is a tropical ideal by the discussion after \cite[Corollary 2.11]{MR20} and therefore has a finite tropical basis $\mathcal{B'}$. For each element $f'$ in $\mathcal{B'}$ choose an $f\in I$ with $\ph(f)=f'$. Let $\calb_0$ be the collection of these $f$s.
    
    Then $\calb=\calb_0\cup\calb_1$ is a finite tropical basis for $E$. 
    \exEnd
\end{example}

A related notion was introduced in \cite{SN24} where they say that $E$ has a finite congruence tropical basis if there are $(f_1,g_1),\ldots,(f_r,g_r)\in E$ such that $V(E)=V((f_1,g_1),\ldots,(f_r,g_r))$. In general, we have that $E$ has a finite tropical basis (in our sense) if and only if $E$ and $E_{\B}$ both have finite congruence tropical bases (in the sense of \cite{SN24}), where $E_{\B}$ is the pushforward of $E$ along the map $\base[\Mon]\to\B[\Mon]$.

Our goal in this section is to prove the following theorem. It gives us a geometric interpretation of what it means for a prime congruence to contain a given congruence, analogous to the fact that an ideal $I$ in a polynomial ring over a field is contained in a prime ideal $P$ if and only if the algebraic set defined by $I$ contains the variety of $P$. 
Proving this result will require technical tools using the filters $\wtcalf_P$.

\begin{theorem}\label{thm:PrimeContainsCongruence}
Let $E$ and $P$ be congruences on $\base[\Mon]$ with $P$ prime. If there is a flag of cones $\calc_\bullet$ such that $P=P_{\calc_\bullet}$ and $\calc_j\subseteq\wt{V}(E)$ for all $j$ then $P\supseteq E$.
If $E$ has a finite tropical basis then the converse is true as well.    
\end{theorem}

\begin{remark}\label{rmk:FiniteTropBasisNeeded}
The hypothesis that $E$ has a finite tropical basis is needed for the converse to be true. To see this, consider the congruence $E$ on $\T[x^{\pm1}]$ generated by the pairs $(t^a+x, t^a)$ for $a>0$. Then $$\wt{V}(E) = \dcap_{\substack{a\in\R\\a>0}} \wt{V}\big((t^a+x, t^a)\big) = \dcap_{\substack{a\in\R\\a>0}} \{(r,x)\in\R_{\geq0}\times\R\,:\,x\leq ra\} = \R_{\geq0}\times(-\infty,0_{\R}].$$
We let $P$ be the prime congruence on $\T[x^{\pm1}]$ defined by the matrix $\begin{pmatrix}
1&0\\0&1
\end{pmatrix}$; direct evaluation of each $t^a+x$ and $t^a$ shows that $E\subseteq P$. But no flag of cones contained in $\wt{V}(E)=\R_{\geq0}\times(-\infty,0_{\R}]$ gives $P$. 
\end{remark}

Before working toward the proof of Theorem~\ref{thm:PrimeContainsCongruence}, we note the following corollary %\orange{which is a natural consequence of Theorem~\ref{thm:PrimeContainsCongruence}, but we do not use it for the subsequent results.}
because it is a natural consequence of Theorem~\ref{thm:PrimeContainsCongruence}; we do not use it in the remainder of this paper.

\begin{coro}
Let $E$ be any congruence on $\base[\Mon]$ that has a finite tropical basis. Then the map sending a flag of cones $\calc_\bullet$ to $P_{\calc_\bullet}$ induces a bijection from the set of flags of cones contained in $\wt{V}(E)$ modulo $\Gamma$-local equivalence\footnote{Two flags of cones are $\Gamma$-locally equivalent if they have the same $\Gamma$-admissible neighborhoods.} to the set of prime congruences on $\base[\Mon]/E$.
\end{coro}\begin{proof}
This follows immediately from \cite[Corollary~3.38]{FM23} and Theorem~\ref{thm:PrimeContainsCongruence}.
\end{proof}

The idea of the proof of Theorem~\ref{thm:PrimeContainsCongruence} is as follows. Given a flag of cones $\calc_\bullet$ as in the statement and $(f,g)\in E$, we want to show $(f,g)\in P$. If we pick $w_i\in\calc_i\sdrop\calc_{i-1}$ for $i=0,\ldots,k$ then evaluation of $f$ and $g$ at $P$ is given by evaluating each of their monomials using the matrix $\begin{pmatrix}
w_0\\w_1\\\vdots\\w_k
\end{pmatrix}$. If we could reduce to the case where $f=m$ and $g=\mu$ are actually monomials, then the fact that $w_i\in\wt{V}(E)\subseteq\wt{V}\big((m,\mu)\big)$ would tell us that $|m|_P=|\mu|_P$, i.e., $(m,\mu)\in P$.

While the reduction described above is not possible, there is another reduction that does work and allows us to focus on monomials. The following lemmas facilitate that reduction.

\begin{lemma}\label{lemma:GrobnerCellOfFContainingP}
Let $P$ be a prime congruence on $\base[\Mon]$ and let $f\in\base[\Mon]$. Then there is a $\Gamma$-admissible cone $L\in\wtcalf_P$ and a term $m$ of $f$ such that $|f|_P=|m|_P$ and $\wt{f}(w)=\wt{m}(w)$ for all $w\in L$.
\end{lemma}\begin{proof}
Let $\tau=\tau_P$ and let $m_0,\ldots,m_n$ be the monomials of $f$, so $f=\dsum_{j=0}^n m_j$. 
Because $\kappa(P)$ is totally ordered, there is a $j$ for which $|m_j|_P$ is as large as possible. We may assume without loss of generality is $j=0$. 

Let $L:=\dcap_{j=1}^n\wt{R}_\tau(m_0,m_j)$; $L$ is a $\Gamma$-admissible cone because $m_0,\ldots,m_n$ are monomials. If $w\in L$ then $\wt{m}_0(w)\geq\wt{m}_j(w)$ for all $j$ so $\wt{f}(w)=\wt{m}_0(w)$. Thus, setting $m=m_0$, we have the result.
\end{proof}

\begin{lemma}\label{lemma:ConeInFilterAndVariety}
Let $P$ be a prime congruence on $\base[\Mon]$ and say $(f,g)\in P$. Then there is a $\Gamma$-admissible cone $\wt{U}\in\wtcalf_P$ that is contained in $\wt{V}\big((f,g)\big)$.
\end{lemma}\begin{proof}
Let $\tau=\tau_P$. By Lemma~\ref{lemma:GrobnerCellOfFContainingP} there is a $\Gamma$-admissible cone $L_f\in\wtcalf_P$ and a term $m_f$ of $f$ such that $|f|_P=|m_f|_P$ and $\wt{f}(w)=\wt{m}_f(w)$ for all $w\in L_f$. Similarly, we get $L_g$ and $m_g$. Since $(f,g)\in P$, we have $|m_f|_P=|f|_P=|g|_P=|m_g|_P$. Thus $L':=\wt{R}_\tau(m_f,m_g)\cap\wt{R}_\tau(m_g,m_f)\in\wtcalf_P$; since $m_f,m_g$ are monomials, $L'$ is a $\Gamma$-admissible cone. 

Let $\wt{U}:=L'\cap L_f\cap L_g\in\wtcalf_P$, which is a $\Gamma$-admissible cone because $L', L_f$, and $L_g$ are.
If $w\in\wt{U}$ then $\wt{f}(w) =\wt{m}_f(w) =\wt{m}_g(w)=\wt{g}(w)$ so $w\in\wt{V}\big((f,g)\big)$. Thus $\wt{U}\subseteq\wt{V}\big((f,g)\big)$.
\end{proof}

\begin{lemma}\label{lemma:CuttingFlagByNbhd}
Let $\cald_\bullet=(\cald_{-1}\leq\cald_0\leq\cdots\leq\cald_k)$ be a flag of cones in $\R_{\geq0}\times(N_{\R}/\tau)$ and let $\wt{U}$ be a ($\Gamma$-admissible) neighborhood of $\cald_\bullet$. Then letting $\calc_j:=\cald_j\cap\wt{U}$ for $j=-1,\ldots,k$ defines a flag of cones $\calc_\bullet$ such that $P_{\calc_\bullet}=P_{\cald_\bullet}$.
\end{lemma}
\begin{proof}
Note that the $\calc_j$ are cones and $\calc_{j-1}\leq\calc_{j}$ for $j\geq0$. So, to show that $\calc_\bullet$ is a flag of cones, we just need to show that $\dim(\calc_j)=j+1$. Since $\calc_j\subseteq\cald_j$, $\dim(\calc_j)\leq\dim(\cald_j)=j+1$. We also have $\dim(\calc_{-1})=0$ because $\calc_{-1}=\cald_{-1}$ is the origin. So, for the rest of this proof, we only consider $j\geq 0$.

Because $\wt{U}$ is a neighborhood of $\cald_\bullet$, for all $j$ we have $\wt{U}\cap\relint\cald_j\neq\emptyset$. Pick $w_j\in\wt{U}\cap\relint\cald_j$, so $w_j\in\calc_j$. Since $w_j\in\relint\cald_j$, $w_j$ is not in the span of $\cald_{j-1}$. Since $\calc_{j-1}\subseteq\cald_{j-1}$ and $\calc_j$ contains both $\calc_{j-1}$ and $w_j$, we have $\dim(\calc_j)\geq \dim\vspan\calc_{j-1}\cup\{w_j\}=\dim(\calc_{j-1})+1$. Thus a quick induction argument shows that $\dim(\calc_j)\geq j+1$. Since we already have the opposite inequality, we conclude that $\dim(\calc_j)=j+1$.

Continuing with $w_0,\ldots,w_k$ as above, note that $w_j\in\calc_j\sdrop\calc_{j-1}$ so $P_{\calc_\bullet}$ is the prime congruence defined by the matrix $\begin{pmatrix}w_0\\w_1\\\vdots\\w_k\end{pmatrix}$. But we also have that $w_j\in\relint\cald_j$, so the same matrix also gives us the prime congruence $P_{\cald_\bullet}$. Thus $P_{\calc_\bullet}=P_{\cald_\bullet}$.
\end{proof}

\begin{proof}[Proof of Theorem~\ref{thm:PrimeContainsCongruence}]
Suppose we have such a $\calc_\bullet=(\calc_{-1}\leq\calc_0\leq\cdots\leq\calc_k)$ defining $P$; we want to show $E\subseteq P$.

Given $(f,g)\in E$, Lemma~\ref{lemma:GrobnerCellOfFContainingP} tells us that there are $\Gamma$-admissible cones $L_f,L_g\in\wtcalf_P$ and terms $m_f$ and $m_g$ of $f$ and $g$, respectively, such that $|f|_P=|m_f|_P$, $\wt{f}(w)=\wt{m}_f(w)$ for all $w\in L_f$, $|g|_P=|m_g|_P$, and $\wt{g}(w)=\wt{m}_g(w)$ for all $w\in L_g$. Setting $L:=L_f\cap L_g$ we have that $L\in\wtcalf_P$ is a $\Gamma$-admissible cone so $L$ is a $\Gamma$-admissible neighborhood of $\calc_\bullet$. Thus, Lemma~\ref{lemma:CuttingFlagByNbhd} tells us that setting $\cald_j:=\calc_j\cap L$ defines a flag of cones $\cald_\bullet$ with $P_{\cald_\bullet}=P_{\calc_\bullet}=P$.

Note that, by the definition of $L$, if $w\in L$ then $\wt{f}(w)=\wt{m}_f(w)$ and $\wt{g}(w)=\wt{m}_g(w)$. So, if $w\in \wt{V}(E)\cap L \subseteq \wt{V}\big((f,g)\big)\cap L$, then $\wt{m}_f(w)=\wt{f}(w)=\wt{g}(w)=\wt{m}_g(w)$.

For each $j\geq0$ pick $w_j\in\cald_j\sdrop\cald_{j-1}$. So $P=P_{\cald_\bullet}$ is given by the matrix $\begin{pmatrix}w_0\\w_1\\\vdots\\w_k\end{pmatrix}$. For each $j$ we have $w_j\in\cald_j=\calc_j\cap L\subseteq\wt{V}(E)\cap L$ so $\wt{m}_f(w_j)=\wt{m}_g(w_j)$. Thus evaluating $m_f$ and $m_g$ using the matrix $\begin{pmatrix}w_0\\w_1\\\vdots\\w_k\end{pmatrix}$ give the same result, i.e., $|m_f|_P=|m_g|_P$. Thus $|f|_P=|m_f|_P=|m_g|_P=|g|_P$ so $(f,g)\in P$. Since we showed this for an arbitrary $(f,g)\in E$, we conclude that $E\subseteq P$.

Now suppose that there are $(f_1,g_1),\ldots,(f_r,g_r)\in E$ such that $\wt{V}(E) = \wt{V}\big((f_1,g_1),\ldots,(f_r,g_r)\big)$ and that $E\subseteq P$. Since $(f_i,g_i)\in E\subseteq P$, Lemma~\ref{lemma:ConeInFilterAndVariety} tells us that there is a $\Gamma$-admissible cone $\wt{U}_i\in\wtcalf_P$ that is contained in $\wt{V}\big((f_i,g_i)\big)$. Letting $\wt{U}:=\dcap_{i=1}^r\wt{U}_i$ we have that $\wt{U}\in\wtcalf_P$ is a $\Gamma$-admissible cone. 

Fix any flag of cones $\cald_\bullet$ with $P=P_{\cald_\bullet}$. So 
by Theorem~\ref{thm:TheFilterHasFlagMeaning} $\wt{U}$ is a $\Gamma$-admissible neighborhood of $\cald_\bullet$. Lemma~\ref{lemma:CuttingFlagByNbhd} now tells us that setting $\calc_j:=\cald_j\cap\wt{U}$ defines a flag of cones $\calc_\bullet$ such that $P_{\calc_\bullet}=P_{\cald_\bullet}=P$. For any $j$ we have
$$ \calc_j \subseteq\wt{U}=\dcap_{i=1}^r\wt{U}_i \subseteq\dcap_{i=1}^r\wt{V}\big((f_i,g_i)\big)=\wt{V}(E).$$ \par\nopagebreak\vspace{-2\baselineskip}\mbox{}
\end{proof}

\begin{remark}
In the example given in Remark~\ref{rmk:FiniteTropBasisNeeded}, while the second row of the matrix, corresponding to $\calc_1$, is not contained in $\wt{V}(E)$, the first row, corresponding to $\calc_0$, is contained in $\wt{V}(E)$. We now use the method at the end of the proof of Theorem~\ref{thm:PrimeContainsCongruence} to show that this happens more generally.
\end{remark}

\begin{coro}
Let $E\subseteq P$ be congruences on $\base[\Mon]$ with $P$ prime. For any flag $\calc_\bullet=(\calc_{-1}\leq\calc_0\leq\cdots\leq\calc_k)$ (with $k\geq0$) with $P=P_{\calc_\bullet}$, we have $\calc_{-1}, \calc_0\subseteq\wt{V}(E)$.
\end{coro}\begin{proof}
Since $\calc_{-1}=\{(0,0)\}$, $\calc_{-1}\subseteq\wt{V}(E)$. 

Now suppose $(f,g)\in E$; we want to show that $\wt{f}$ and $\wt{g}$ agree on $\calc_0$, i.e., $\calc_0\subseteq\wt{V}\big((f,g)\big)$.
Since $(f,g)\in E\subseteq P$, Lemma~\ref{lemma:ConeInFilterAndVariety} tells us that there is a $\Gamma$-admissible cone $\wt{U}\in\wtcalf_P$ that is contained in $\wt{V}\big((f,g)\big)$. 
By Theorem~\ref{thm:TheFilterHasFlagMeaning}, $\wt{U}$ is a $\Gamma$-admissible neighborhood of $\calc_\bullet$, so Lemma~\ref{lemma:CuttingFlagByNbhd} now tells us that setting $\cald_j:=\calc_j\cap\wt{U}$ defines a flag of cones $\cald_\bullet$ such that $P_{\cald_\bullet}=P_{\calc_\bullet}$. In particular, $\dim(\calc_0\cap\wt{U})=1=\dim(\calc_0)$, so $\calc_0\cap\wt{U}=\calc_0$, and thus $\calc_0 \subseteq \wt{U} \subseteq \wt{V}\big((f,g)\big)$, as desired.
\end{proof}

\section{Gleichstellensatz}
In this section we provide several important consequences of Theorem~\ref{thm:PrimeContainsCongruence}.
%Let $\base$ be a sub-semifield of $\T$ and let $\Mon$ be a toric monoid corresponding to a cone $\sigma\subseteq N_{\R}$.

In applying Theorem~\ref{thm:PrimeContainsCongruence} to ideals we must proceed with caution -- it is possible for an ideal $I$ to have a finite tropical basis but for $\Bend(I)$ to not have a finite tropical basis, as the following example shows.

\begin{example}
Let $I\subseteq\T[x,y,z]$ be the ideal generated by the polynomials
\begin{align*}
g_{xy}&=1+x^2+y^2+z^2+txy\\
g_{xz}&=1+x^2+y^2+z^2+txz\\
g_{yz}&=1+x^2+y^2+z^2+tyz, \text{ and}\\
f_n&=1+x+y+z^n
\end{align*}
for all $n\geq1$.
Routine computation shows that 
\begin{align*}
V(g_{xy})\cap V(g_{xz})\cap V(g_{yz}) &= \{(x,y,z)\in(\R\cup\{-\infty\})^3 \,:\, x\leq -1,\, y\leq -1,\, z=0\}\\
&\ \ \cup\{(x,y,z)\in(\R\cup\{-\infty\})^3 \,:\, x\leq -1,\, y=0,\, z\leq-1\}\\
&\ \ \cup\{(x,y,z)\in(\R\cup\{-\infty\})^3 \,:\, x=0,\, y\leq-1,\, z\leq-1\}
\end{align*}
and
\begin{align*}
V(f_n)&=\{(x,y,z)\in(\R\cup\{-\infty\})^3 \,:\, x=0,\, y\leq 0,\, z\leq0\}\\
&\ \ \cup\{(x,y,z)\in(\R\cup\{-\infty\})^3 \,:\, x\leq0,\, y=0,\, z\leq0\}\\
&\ \ \cup\{(x,y,z)\in(\R\cup\{-\infty\})^3 \,:\, x\leq0,\, y\leq0,\, z=0\}\\
&\ \ \cup\{(x,y,z)\in(\R\cup\{-\infty\})^3 \,:\, x=y\geq0,\, x\geq nz\}\\
&\ \ \cup\{(x,y,z)\in(\R\cup\{-\infty\})^3 \,:\, x=nz\geq0,\, x\geq y\}\\
&\ \ \cup\{(x,y,z)\in(\R\cup\{-\infty\})^3 \,:\, y=nz\geq0,\, y\geq x\}
\end{align*}
In particular, $V(g_{xy})\cap V(g_{xz})\cap V(g_{yz})\subseteq V(f_n)$ for all $n$, so $\{g_{xy},g_{xz},g_{yz}\}$ is a tropical basis for $I$.

We now show that $\Bend(I)$ does not have a finite tropical basis. For this, it suffices to show that there is no finite set $X\subseteq I$ with $\wt{V}(\bend{X})=\wt{V}(\Bend(I))$. 

Letting $\pi_{\B}:\base[\Mon]\to\B[\Mon]$ be the homomorphism induced by sending all nonzero $a\in\base$ to $1_{\B}$, we have that 
% $\wt{V}(\Bend(I))\cap(\{0\}\times (\R\cup\{-\infty\})^3)=\{0\}\times V(\Bend(\pi_{\B}(I)))$. 
$\wt{V}(\Bend(I))\cap(\{0\}\times (\R\cup\{-\infty\})^3)=\{0\}\times V(\pi_{\B}(I))$. 
So it suffices to show that there is no finite set $X\subseteq I$ such that 
% $V(\Bend(\pi_{\B}(I)))=V(\bend(\pi_{\B}(X)))$. 
$V(\pi_{\B}(I))=V(\pi_{\B}(X))$. 
Moreover, by Lemma~\ref{lemma:TropBasisInGeneratingSet}, it suffices to show that there is no finite set $X\subseteq\{g_{xy},g_{xz},g_{yz}\}\cup\{f_n \,:\, n\geq1\}$ such that 
% $V(\Bend(\pi_{\B}(I)))=V(\bend(\pi_{\B}(X)))$.
$V(\pi_{\B}(I))=V(\pi_{\B}(X))$.

Note that 
\begin{align*}
V(\pi_{\B}(g_{xy}))&=V(\pi_{\B}(g_{xz}))=V(\pi_{\B}(g_{yz}))\\ 
&=\{(x,y,z)\in(\R\cup\{-\infty\})^3 \,:\, x=0,\, y\leq 0,\, z\leq 0\}\\
&\ \ \cup\{(x,y,z)\in(\R\cup\{-\infty\})^3 \,:\, x\leq0,\, y=0,\, z\leq0\}\\
&\ \ \cup\{(x,y,z)\in(\R\cup\{-\infty\})^3 \,:\, x\leq0,\, y\leq 0,\, z=0\}\\
&\ \ \cup\{(x,y,z)\in(\R\cup\{-\infty\})^3 \,:\, x=y\geq0,\, z\leq x\}\\
&\ \ \cup\{(x,y,z)\in(\R\cup\{-\infty\})^3 \,:\, x=z\geq0,\, y\leq x\}\\
&\ \ \cup\{(x,y,z)\in(\R\cup\{-\infty\})^3 \,:\, y=z\geq0,\, x\leq y\}
\end{align*}
and $V(\pi_{\B}(f_n))=V(f_n)$. Thus 
\begin{align*}
V(\pi_{\B}(I))&=\{(x,y,z)\in(\R\cup\{-\infty\})^3 \,:\, x=0,\, y\leq 0,\, z\leq 0\}\\
&\ \ \cup\{(x,y,z)\in(\R\cup\{-\infty\})^3 \,:\, x\leq0,\, y=0,\, z\leq0\}\\
&\ \ \cup\{(x,y,z)\in(\R\cup\{-\infty\})^3 \,:\, x\leq0,\, y\leq 0,\, z=0\}.
\end{align*}
But if $X$ is any finite subset of $\{g_{xy},g_{xz},g_{yz}\}\cup\{f_n \,:\, n\geq1\}$ and $n_0$ is the largest $n$ such that $f_n\in X$ (or $0$ if there is no such $n$), then
\begin{align*}
V(\pi_{\B}(X))&\supseteq \{(x,y,z)\in(\R\cup\{-\infty\})^3 \,:\, x=0,\, y\leq 0,\, z\leq0\}\\
&\ \ \cap\{(x,y,z)\in(\R\cup\{-\infty\})^3 \,:\, x\leq0,\, y=0,\, z\leq0\}\\
&\ \ \cap\{(x,y,z)\in(\R\cup\{-\infty\})^3 \,:\, x\leq0,\, y\leq0,\, z=0\}\\
&\ \ \cap\{(x,y,z)\in(\R\cup\{-\infty\})^3 \,:\, x=y\geq0,\, x\geq n_0z\},
\end{align*}
so $V(\pi_{\B}(X))\supsetneq V(\pi_{\B}(I))$.
\exEnd
\end{example}

\begin{coro}\label{coro:pre-Nullstellensatz}

Let $\base\subseteq\T$ be a sub-semifield and let $\Mon$ be a toric monoid. Let $E$ be a congruence on $\base[\Mon]$ with a finite tropical basis $\calB$. Then the finitely generated congruence $C=\angbra{\calB}$ satisfies $\sqrt{E}=\sqrt{C}$. In particular, if $I\subseteq\base[\Mon]$ is a tropical ideal then there is a finitely generated ideal $J$ such that $\sqrt{\Bend(I)}=\sqrt{\Bend(J)}$. 
\end{coro}\begin{proof}

It suffices to show that, for any prime congruence $P$, $P\supseteq E$ if and only if $P\supseteq C$. We have $\wt{V}(C)=\wt{V}(E)$. 
By Theorem~\ref{thm:PrimeContainsCongruence}, $P\supseteq E$ if and only if there is a $\calc_\bullet$ with $P=P_{\calc_\bullet}$ and $\calc_j\subseteq\wt{V}(E)$ if and only if there is a $\calc_\bullet$ with $P=P_{\calc_\bullet}$ and $\calc_j\subseteq\wt{V}(C)$ if and only if $P\supseteq C$.
\end{proof}

The next corollary states that the radical of a congruence $E$ with finite tropical basis is determined by its variety, as rank one primes containing $E$ correspond to points on the variety $\wt{V}(E)$.

\begin{coro}\label{coro:TropBasisImpliesRankOnePrimesEnough}
Let $\base\subseteq\T$ be a sub-semifield and let $\Mon$ be a toric monoid. 
Let $E$ be a congruence on $\base[\Mon]$ with a finite tropical basis. Then the radical of $E$ is the intersection of the rank 1 prime congruences containing $E$, i.e.,
$$\sqrt{E}=\dcap_{w\in\wt{V}(E)}P_w.$$
\end{coro}\begin{proof}
Let $\calB$ be a tropical basis for $E$ and let $C=\angbra{\calB}$. By Theorem~\ref{thm:PrimeContainsCongruence}, $P\supseteq C$ if and only if there is a flag of cones $\calc_\bullet$ such that $P=P_{\calc_\bullet}$ and $\calc_j\subseteq\wt{V}(C)$ for all $j$. So for any rank-one prime $P_w$, we have $P_w\supseteq C$ if and only if $w\in\wt{V}(C)$. By \cite[Theorem 5.4]{JM17} we know that the radical of a finitely generated congruence is the intersection of the rank-one primes containing it. Thus $$\bigcap_{w\in \wt V(E)} P_w = \bigcap_{w\in \wt{V}(C)} P_w = \sqrt{C} = \sqrt{E},$$ where the last equality is true by Corollary~\ref{coro:pre-Nullstellensatz}.
\end{proof}

In \cite{JM17} the authors prove a tropical Nullstellensatz for finitely generated congruences. The statement below extends that result to congruences whose variety can be cut out by finitely many pairs, such as the bend congruences of tropical ideals.  

\begin{coro}[Gleichstellensatz for congruences with a finite tropical basis]\label{coro: infiniteNull}
Let $\base \subseteq \T$ be a sub-semifield and let $\Mon$ be a toric monoid. Let $E$ be a congruence on $\base[\Mon]$ that has a finite tropical basis. Suppose $(f,g)$ is such that $\wt f(w)=\wt g(w)$ for all $w\in \wt V(E)$. Then there are $i\in\Z_{\geq0}$ and $h\in\base[\Mon]$ such that $\big[(f+g)^i+h\big](f,g)\in E$.
\end{coro}

\begin{proof}
By hypothesis, $(f,g)\in\dcap_{w\in\wt{V}(E)}P_{w}$, which orollary~\ref{coro:TropBasisImpliesRankOnePrimesEnough} says is $\sqrt{E}$. So Corollary~\ref{coro:radical_by_generalized_power} tells us that there are $i\in\Z_{\geq0}$ and $h\in\base[\Mon]$ such that $\big[(f+g)^i+h\big](f,g)\in E$.
\end{proof}

The following example demonstrates that in Theorem~\ref{thm:PrimeContainsCongruence}, Corollary~\ref{coro:TropBasisImpliesRankOnePrimesEnough}, and Corollary~\ref{coro: infiniteNull} the finite tropical basis hypothesis is necessary.

\begin{example}
Let $\base=\T$ and $\Mon=\N$ so $\base[\Mon]=\T[x]$. Let $E$ be the prime congruence on $\T[x]$ given by the matrix $\begin{pmatrix}
1&0\\
0&-1
\end{pmatrix}$.
Then $\wt{V}(E)=\R_{\geq0}\times\{0\}$. 
So, if we let $f(x)=x$ and $g(x)=1$ then $\wt{f}(w)=\wt{g}(w)$ for all $w\in\wt{V}(E)$. Now suppose there are some $i\in\Z_{\geq0}$ and $h\in\T[x]$ such that $[(1+x)^i+h](1,x)\in E$. Since $E$ is prime but $(1,x)\notin E$, this implies that $\left((1+x)^i+h,0\right)\in E$. But $E$ has trivial ideal-kernel, so this is impossible.
\exEnd
\end{example}

\section{Minimal primes containing a congruence}

Let $\base$ be a subsemifield of $\T$, let $\Gamma$ be the corresponding additive subgroup of $\R$, and let $\Mon$ be a toric monoid. 

We know from a straightforward generalization of \cite[Theorem 4.14]{JM17} that every prime congruence of $\base[\Mon]$ contains a prime with trivial ideal-kernel. In this section we show that the same is true for prime congruences of $\base[\Mon]/E$, when $E$ satisfies certain geometric conditions. In particular, it is true for the coordinate semirings of tropicalizations of varieties under very mild geometric assumptions. We start with some motivating examples of prime congruences with trivial ideal-kernel contained in a prime congruence with non-zero ideal-kernel.

\begin{example}\label{ex: noKerPrime1} Consider $A = \mathbb{B}[x, y, z]/E$, where $E = \left< (x + x^2y^2 + z^3, x^2z + x^2 + y)\right>$. Let $P$ be the prime congruence corresponding to the matrix $\begin{pmatrix} -\infty & -\infty & -\infty \end{pmatrix}$, i.e., the prime generated by $(x, 0)$, $(y, 0)$, and $(z, 0)$. Note that $P$ contains $E$. The primes with trivial ideal-kernel that contain $E$ and are contained in $P$ correspond to the following matrices:
   \begin{align*}
       &C_1 = \begin{pmatrix} \alpha & -1 & -1/3 \end{pmatrix}, \alpha \leq -1, \quad
       C_2 = \begin{pmatrix} -1 & 0 & 0 \\ 0 & -1 & -1/3 \end{pmatrix} \\ 
       &C_3 = \begin{pmatrix} -1 & -1 & \gamma \end{pmatrix}, \gamma < 1, \quad \quad \ \ \ 
       C_4 = \begin{pmatrix} 0 & 0 & -1 \\ -1 & -1 & 0 \end{pmatrix}
    \end{align*} \par\nopagebreak\vspace{-1.5\baselineskip}\mbox{}
    \exEnd
\end{example}

\begin{example}\label{ex: noKerPrime2} Consider $A = \mathbb{T}[x, y, z]/E$, where $E$ is the congruence generated by the bend relations of the polynomial $f = x^2 + txy + y^2 + x^2y^2$, i.e., $E = \left< \bend(f) \right>$. Let $P$ be the prime corresponding to the matrix $\begin{pmatrix} 1 & -\infty & -\infty \end{pmatrix}$, i.e., the prime generated by $(x, 0)$ and $(y, 0)$. Note that $P$ contains $E$.

A prime with trivial ideal-kernel that contains $E$ and is contained in $P$ corresponds to the matrix
$\begin{pmatrix} 0& -1 & -1 \\ 1 & 0 & -1\end{pmatrix}$. This can be written as $\begin{pmatrix} 0& v \\ 1 & \what{w}\end{pmatrix}$ where $v=(-1, -1)$ is a point on the recession of $V(E)$ that "points toward" $(-\infty,-\infty)$ and $\what{w}=(0,-1)$ is a point in the relative interior of a maximal cell of $V(E)$ whose closure contains $(-\infty,-\infty)$. The proof of Theorem~\ref{thm:resolving-primes} will proceed by finding similar $v$s and $\what{w}$s.

Furthermore, if $k$ is a field with valuation $v:k\to \T$ and $F\in k[x,y]$ is any polynomial whose tropicalization is $f=x^2 + txy + y^2 + x^2y^2$, then the above example works verbatim if we replace $E$ with $ \Bend(\trop\left<F\right>)$. \exEnd
\end{example}

We now state the main result of this section.

\begin{thm}\label{thm:resolving-primes}
Let $E$ be a congruence on $\base[\Mon]$ with a finite tropical basis. 
Suppose that $\wt{V}(E)$ is the closure of $\wt{V}(E)\cap(\R_{\geq0}\times N_{\R})$ in $\R_{\geq0}\times N_{\R}(\sigma)$. 
Then every minimal prime congruence of $\base[\Mon]/E$ has trivial ideal-kernel.
\end{thm}

To prove Theorem~\ref{thm:resolving-primes} we consider a prime congruence $P$ of $\base[\Mon]$ containing $E$, then construct a prime $Q$ with trivial ideal-kernel, and show that $E\subseteq Q\subseteq P$. 
%\orange{To show the first inclusion $E\subseteq Q$ we first prove Lemma~\ref{lemma:InitialFormsAddingVectors} which is an analogue of \cite[Lemma 2.4.6]{MS}. We note that it is not enough to show the analogous statement for some small enough $\epsilon$, but understand exactly for what set of parameters this statements is true, as in Corollary~\ref{coro:InitialFormsAddingManyVectors} we apply this result iteratively.}

In order to show that $(f,g)\in E$ implies $(f,g)\in Q$, we consider the $Q$-leading terms of $f$ and $g$. 
Once we do so, we will prove Lemma~\ref{lemma:InitialFormsAddingVectors}, which is an analogue of \cite[Lemma 2.4.6]{MS}. However, it will not be enough for us to show the analogous statement for all small enough $\varepsilon$; we need to understand exactly for what set of parameters this statement is true so that, when we apply this result iteratively in Corollary~\ref{coro:InitialFormsAddingManyVectors} we get a result that is true for a sufficiently large set of parameters.

We first define a very general type of initial form. Let $\displaystyle f = \sum_{u\in \Mon} t^{a_u}\chi^u\in \base[\Mon]$. The support of $f$ is the set $\supp{f} = \{ u \in \Mon : a_u \neq -\infty\}$.
Given $w \in \R_{\geq0}\times N_{\R}(\sigma)$, we define the following initial form: $$ \whin_w(f) = \dsum_{\substack{u\in \supp{f} \\ \left< (a_u, u), w\right> \text{ is maximized}}} t^{a_u}\chi^u.$$ Note that, since we do not require $w\in\R_{\geq0}\times N_{\R}$, $\angbra{(a_u,u),w}$ may sometimes be $-\infty$. 
 
Our initial form is different from the definition that Maclagan and Rinc\'on given in \cite[Definition 3.1]{MR18}; their initial form $\mathrm{in}_w(f)$ is in $\mathbb{B}[\Mon]$. The relationship between the two definitions is $\mathrm{in}_w(f) = \pi_\mathbb{B} (\whin_w(f)),$ where $\pi_\mathbb{B}$ is the map that sends all non-zero coefficients of the polynomial to $1_\mathbb{B}$. %\orange{In particular, $\whin_w(f)$ is not an initial form in the Gr\"obner theoretic sense as we keep the coefficients of the maximized monomials, and we do not apply $\whin_w$ to an ideal, just a polynomial.}
In particular, $\whin_w(f)$ is not an initial form in the Gr\"obner theoretic sense as we keep the coefficients of the maximized monomials. In particular, we never apply $\whin_w$ to an ideal, only a polynomial.

We now define an initial form with respect to any prime congruence $P$ on $\base[\Mon]$, which recovers the previous definition if the prime is defined by a matrix that has a single row.
$$ \whin_P(f) = \dsum_{\substack{u\in \supp{f} \\ |t^{a_u}\chi^u|_P \text{ is maximized}}} t^{a_u}\chi^u.$$
Note that if $P$ is given by $\begin{pmatrix}w_0\\w_1\\\vdots\\w_k\end{pmatrix}$ then $\whin_P(f)=\whin_{w_k}(\whin_{w_{k-1}}(\cdots\whin_{w_0}(f)\cdots))$.

\begin{lemma}\label{lemma:InitialFormsAddingVectors}
Fix $f\in\base[\Mon]$ and $v\in\R_{\geq0}\times N_{\R}$. There is a finite set $\Xi\subset \Gamma\times\Mon$  and, for each $m\in\Xi$, a positive number $\calV_m$ such that the following holds. For any $w\in \R_{\geq0}\times N_{\R}$ and any real number $\largeNum$, if $\largeNum \calV_m>-\left(\wt{\whin_v(f)}(w)-\angbra{m,w}\right)$ for all $m\in\Xi$ then $\wt{f}(w+\largeNum v)=\wt{\whin_v(f)}(w)+\largeNum\wt{f}(v)$ and $\whin_{w+\largeNum v}(f)=\whin_{w}\left(\whin_v(f)\right)$. The same is true if we allow $v,w\in \R_{\geq0}\times N_{\R}(\sigma)$ so long as $v$ and $w$ are in the same stratum.
\end{lemma}
\begin{proof}
Write $f=\dsum_{u\in\Mon}t^{a_u}\chi^u$. Let $\Xi=\left\{(a_u,u)\;:\; u\in\supp(f)\,\&\, u\notin\supp\left(\whin_{v}(f)\right)\right\}$. For any $m=(a_u,u)\in\Xi$, we have $\wt{f}(v)>\angbra{m,v}$ because $u\in\supp(f)$ but $u\notin\supp\left(\whin_{v}(f)\right)$. We let $\calV_m:=\wt{f}(v)-\angbra{m,v}>0$.

Now fix $w\in\R_{\geq0}\times N_{\R}$ and suppose that $\largeNum\calV_m>-\left(\wt{\whin_v(f)}(w)-\angbra{m,w}\right)$ for all $m\in\Xi$. We will show that $\wt{f}(w+\largeNum v)=\wt{\whin_v(f)}(w)+\largeNum\wt{f}(v)$ and $\whin_{w+\largeNum v}(f)=\whin_{w}\left(\whin_v(f)\right)$ by evaluating each term of $f$ at $w+\largeNum v$. So consider $u\in\supp(f)$ and let $m=(a_u,u)$. 

If $u\notin\supp\left(\whin_v(f)\right)$ then $m=(a_u,u)\in\Xi$ and so 
$$\largeNum\left(\wt{f}(v)-\angbra{m,v}\right)>-\left(\wt{\whin_v(f)}(w)-\angbra{m,w}\right),$$
i.e., $\angbra{m,w+\largeNum v}<\wt{\whin_v(f)}(w)+\largeNum \wt{f}(v)$.

Say $u\in\supp\left(\whin_v(f)\right)$. We first consider the case when $u\notin \supp\left(\whin_w\left(\whin_v(f)\right)\right)$. Then $\angbra{m,w+\largeNum v}=\angbra{m,w}+\largeNum\wt{f}(v)<\wt{\whin_v(f)}(w)+\largeNum\wt{f}(v)$. Finally, if $u\in \supp\left(\whin_w\left(\whin_v(f)\right)\right)$, then $\angbra{m,w+\largeNum v}=\wt{\whin_v(f)}(w)+\largeNum\wt{f}(v)$.

Thus $\wt{f}(w+\largeNum v)=\max\{\angbra{m,w+\largeNum v}\;:\;u\in\supp(f)\,\&\,m=(a_u,u)\}=\wt{\whin_v(f)}(w)+\largeNum\wt{f}(v)$ and this maximum occurs at $m=(a_u,u)$ if and only if the term $t^{a_u}\chi^u$ occurs in $\whin_{w+\largeNum v}(f)=\whin_w\left(\whin_v(f)\right)$.
\end{proof}

\begin{coro}\label{coro:InitialFormsAddingManyVectors}

Fix $\xi_0,\xi_1,\ldots,\xi_k\in \R_{\geq0}\times N_{\R}$ and $f_1,\ldots,f_\ell\in \base[\Mon]$. Then there is a full-dimensional polyhedron $\frakQ$ in $\R^k$ such that, if $(\largeNum_{1},\largeNum_2,\ldots,\largeNum_{k-1},\largeNum_{k})$ is in the interior of $\frakQ$ then we have 
$$\whin_{\xi_0}\left(\whin_{\xi_{1}}\left(\cdots\whin_{\xi_{k-1}}\left(\whin_{\xi_k}(f_i)\right)\cdots\right)\right)=\whin_{\xi_0+\largeNum_{1} \xi_{1}+\cdots+\largeNum_{k-1} \xi_{k-1}+\largeNum_{k} \xi_{k}}(f_i)$$
for $i=1,\ldots,\ell$. 
Moreover, if  $0\ll\largeNum_{1}\ll\largeNum_{2}\ll\cdots\ll\largeNum_{k-1}\ll\largeNum_{k}$, then $(\largeNum_{1},\largeNum_2,\ldots,\largeNum_{k-1},\largeNum_{k})$ is in the interior of $\frakQ$.
\end{coro}

\begin{proof}

We will prove a slightly stronger statement: under the above hypotheses, if we define $h_{i,j}$ for $1\leq i\leq \ell$ and $0\leq j\leq k$ by $h_{i,k}=f_i$ and $h_{i,j-1}=\whin_{\xi_j}\left(\cdots\whin_{\xi_k}(f_i)\cdots\right)$
for $1\leq j\leq k$, then we also have 
$$\wt{f_i}\big(\xi_0+\largeNum_{1} \xi_{1}+\cdots+\largeNum_{k-1} \xi_{k-1}+\largeNum_{k} \xi_{k}\big)=\wt{h_{i,0}}(\xi_0)+\largeNum_1\wt{h_{i,1}}(\xi_1)+\largeNum_2\wt{h_{i,2}}(\xi_2)+\cdots+\largeNum_k\wt{h_{i,k}}(\xi_k).$$

We prove this by induction on $k$. The statement for $k=1$ is immediate from Lemma~\ref{lemma:InitialFormsAddingVectors}. 

Now assume the statement is true for $k-1$; we want to show it is true for $k$. We first consider the case $\ell=1$ which is where we have a single $f\in\base[\Mon]$. Let $g=\whin_{\xi_{k}}(f)$ and define $h_{k-1}:=g=\whin_{\xi_{k}}(f)$ and, for $1\leq j\leq k-1$, 
$$h_{j-1}:=\whin_{\xi_j}\left(\cdots \whin_{\xi_{k-1}}(g) \cdots\right)=\whin_{\xi_j}\left(\cdots \whin_{\xi_{k-1}}\left(\whin_{\xi_{k}}(f)\right) \cdots\right).$$
By the inductive hypothesis there is a full-dimensional polyhedron $\frakP$ in $\R^{k-1}$ such that, if $(\largeNum_1,\ldots,\largeNum_{k-1})$ is in the interior of $\frakP$ then 
$$\whin_{\xi_0}\left(\whin_{\xi_{1}}\left(\cdots\whin_{\xi_{k-1}}(g)\cdots\right)\right)=\whin_{\xi_0+\largeNum_{1} \xi_{1}+\cdots+\largeNum_{k-1} \xi_{k-1}}(g)$$
and
$$\wt{g}\big(\xi_0+\largeNum_{1} \xi_{1}+\cdots+\largeNum_{k-1} \xi_{k-1}\big)=
\wt{h_{0}}(\xi_0)+\largeNum_1\wt{h_{1}}(\xi_1)+\largeNum_2\wt{h_{2}}(\xi_2)+\cdots+\largeNum_{k-1}\wt{h_{k-1}}(\xi_{k-1})$$
and if $0\ll\largeNum_{1}\ll\largeNum_{2}\ll\cdots\ll\largeNum_{k-1}$ then $(\largeNum_1,\ldots,\largeNum_{k-1})$ is in the interior of $\frakP$.

By Lemma~\ref{lemma:InitialFormsAddingVectors} applied with $v=\xi_k$, there is a finite set $\Xi\subset \Gamma\times\Mon$ and, for each $m\in\Xi$, a positive number $\calV_m$ such that the following holds. For any $w\in \R_{\geq0}\times N_{\R}$, if $\largeNum_k \calV_m>-\left(\wt{\whin_{\xi_k}(f)}(w)-\angbra{m,w}\right)$ for all $m\in\Xi$ then $\wt{f}(w+\largeNum_k\xi_k)=\wt{\whin_{\xi_k}(f)}(w)+\largeNum_k\wt{f}(\xi_k)$ and $\whin_{w+\largeNum_k \xi_k}(f)=\whin_{w}\left(\whin_{\xi_k}(f)\right)$.
That is, 
if $\largeNum_k \calV_m>-\left(\wt{g}(w)-\angbra{m,w}\right)$ for all $m\in\Xi$ then $\wt{f}(w+\largeNum_k\xi_k)=\wt{g}(w)+\largeNum_k\wt{f}(\xi_k)$ and $\whin_{w+\largeNum_k \xi_k}(f)=\whin_{w}(g)$.

Suppose $(\largeNum_1,\ldots,\largeNum_{k-1})$ is in the interior of $\frakP$ and let $w=\xi_0+\largeNum_{1} \xi_{1}+\cdots+\largeNum_{k-1} \xi_{k-1}$. If 
\begin{align*}
\largeNum_k \calV_m&>-\left(\wt{g}(w)-\angbra{m,w}\right)\\
&=-\left(\wt{g}\left(\xi_0+\dsum_{j=1}^{k-1}\largeNum_{j} \xi_{j}\right)-\angbra{m,\xi_0+\dsum_{j=1}^{k-1}\largeNum_{j} \xi_{j}}\right)\\
&=-\left(\wt{h_{0}}(\xi_0)+\dsum_{j=1}^{k-1}\largeNum_j\wt{h_{j}}(\xi_j)-\left(\angbra{m,\xi_0}+\dsum_{j=1}^{k-1} \largeNum_j \angbra{m,\xi_j}\right)\right)\\
&=-\left(\left(\wt{h_{0}}(\xi_0)-\angbra{m,\xi_0}\right)+\dsum_{j=1}^{k-1}\largeNum_j\left(\wt{h_{j}}(\xi_j)-\angbra{m,\xi_j}\right)\right)
\end{align*}
for all $m\in\Xi$, then we have
\begin{align*}
\hspace{-1cm}\wt{f}(\xi_0+\largeNum_{1} \xi_{1}+\cdots+\largeNum_{k-1} \xi_{k-1} +\largeNum_k\xi_k)
&=\wt{f}(w+\largeNum_k\xi_k)
=\wt{g}(w)+\largeNum_k\wt{f}(\xi_k)\\
&=\wt{g}(\xi_0+\largeNum_{1} \xi_{1}+\cdots+\largeNum_{k-1} \xi_{k-1} +\largeNum_k\xi_k)+\largeNum_k\wt{f}(\xi_k)\\
&=\wt{h_{0}}(\xi_0)+\largeNum_1\wt{h_{1}}(\xi_1)+\largeNum_2\wt{h_{2}}(\xi_2)+\cdots+\largeNum_{k-1}\wt{h_{k-1}}(\xi_{k-1})+\largeNum_k\wt{f}(\xi_k)
\end{align*}
and
\begin{align*}
\whin_{\xi_0+\largeNum_{1} \xi_{1}+\cdots+\largeNum_{k-1} \xi_{k-1}+\largeNum_k \xi_k}(f)
&=\whin_{w+\largeNum_k \xi_k}(f)\\
&=\whin_{w}(g)
=\whin_{\xi_0+\largeNum_{1} \xi_{1}+\cdots+\largeNum_{k-1} \xi_{k-1}}(g)\\
&=\whin_{\xi_0}\left(\whin_{\xi_{1}}\left(\cdots\whin_{\xi_{k-1}}(g)\cdots\right)\right)\\
&=\whin_{\xi_0}\left(\whin_{\xi_{1}}\left(\cdots\whin_{\xi_{k-1}}\left(\whin_{\xi_k}(f)\right)\cdots\right)\right)
\end{align*}

Now let 
$$\frakQ:=\left\{
(\largeNum_1,\ldots,\largenum_k)\in\R^k\;:\;
\begin{array}{l}
(\largenum_1,\ldots,\largenum_{k-1})\in\frakP\;\&\; \\
\largeNum_k \calV_m
\geq
-\left(\left(\wt{h_{0}}(\xi_0)-\angbra{m,\xi_0}\right)+\dsum_{j=1}^{k-1}\largeNum_j\left(\wt{h_{j}}(\xi_j)-\angbra{m,\xi_j}\right)\right)\\
\multicolumn{1}{r}{\text{ for all }m\in\Xi}
\end{array}
\right\}$$
and
$$\frakQ^\circ:=\left\{
(\largeNum_1,\ldots,\largenum_k)\in\R^k\;:\;
\begin{array}{l}
(\largenum_1,\ldots,\largenum_{k-1})\text{ is in the interior of }\frakP\;\&\; \\
\largeNum_k \calV_m
>
-\left(\left(\wt{h_{0}}(\xi_0)-\angbra{m,\xi_0}\right)+\dsum_{j=1}^{k-1}\largeNum_j\left(\wt{h_{j}}(\xi_j)-\angbra{m,\xi_j}\right)\right)\\
\multicolumn{1}{r}{\text{ for all }m\in\Xi}
\end{array}
\right\}.$$
Note that $\frakQ$ is a polyhedron in $\R^k$ and, if $\frakQ^\circ$ is nonempty then $\frakQ$ is full-dimensional and $\frakQ^\circ$ is the interior of $\frakQ$. But if $0\ll\largeNum_{1}\ll\largeNum_{2}\ll\cdots\ll\largeNum_{k-1}\ll\largeNum_{k}$, then $(\largeNum_{1},\largeNum_2,\ldots,\largeNum_{k-1})$ is in the interior of $\frakP$ and so, because the $\calV_m$ are all positive, $(\largeNum_{1},\largeNum_2,\ldots,\largeNum_{k-1},\largeNum_{k})$ is in $\frakQ^\circ$. So $\frakQ^\circ$ is nonempty. This completes the proof for the case where $\ell=1$.

Now let $\ell$ be arbitrary. Given $f_1,\ldots,f_\ell\in\base[\Mon]$ let $g_i=\whin_{\xi_k}(f_i)$ for $i=1,\ldots,\ell$. By the inductive hypothesis, there is one polyhedron $\frakP$ that works for $g_1,\ldots,g_\ell$. Then, as above, we get $\frakQ_i$ and $\frakQ_i^\circ$ that work for $f_i$. Now let $\frakQ=\dcap_{i=1}^\ell\frakQ_i$ and $\frakQ^\circ=\dcap_{i=1}^\ell\frakQ_i^\circ$. 
Then $\frakQ$ is a polyhedron in $\R^k$ and, if $\frakQ^\circ$ is nonempty, then $\frakQ$ is full-dimensional and $\frakQ^\circ$ is the interior of $\frakQ$. But if $0\ll\largeNum_{1}\ll\largeNum_{2}\ll\cdots\ll\largeNum_{k-1}\ll\largeNum_{k}$, then $(\largeNum_{1},\largeNum_2,\ldots,\largeNum_{k-1},\largeNum_{k})$ is in each $\frakQ_i^\circ$ and so is in  $\frakQ^\circ$. Thus $\frakQ^\circ$ is nonempty.
\end{proof}

Having analyzed initial forms, we now turn to the hypothesis of Theorem~\ref{thm:resolving-primes} and so investigate closures in $\R_{\geq0}\times N_{\R}(\sigma)$ of subsets of $\R_{\geq0}\times N_{\R}$ and some corresponding polyhedral geometry.

\begin{lemma}\label{lemma:PolyhedronClosureMeansEndOfRay}
Let $\tau\leq\sigma$ and let $w_0,w_1,\ldots,w_k\in\R_{\geq0}\times(N_{\R}/\tau)$. If $L\subseteq \R_{\geq0}\times N_{\R}$ is a cone such that $w_i\in\cl_{\R_{\geq0}\times N_{\R}(\sigma)}(L)$ for $i=0,1,\ldots,k$, then there are $v,\what{w}_0,\what{w}_1,\ldots,\what{w}_k\in L$ such that $\dlim_{\largeNum\to\infty}\what{w}_i+\largeNum v=w_i$ for $i=0,1,\ldots,k$. Moreover, if we let $\pi_\tau:N_{\R}\to N_{\R}/\tau$ be the quotient map, then the sets $L\cap(\id_{\R_{\geq0}}\times\pi_\tau)^{-1}(w_i)$ and $L\cap[\{0_{\R}\}\times(\relint\tau)]$ are nonempty, $\what{w}_i$ can be any element of $L\cap(\id_{\R_{\geq0}}\times\pi_\tau)^{-1}(w_i)$, and $v$ can be any element of $L\cap[\{0_{\R}\}\times(\relint\tau)]$.
\end{lemma}
\begin{proof}

Note that, by restricting our attention to the open set $\R_{\geq0}\times N_{\R}(\tau)\subseteq \R_{\geq0}\times N_{\R}(\sigma)$, we may assume without loss of generality that $\sigma=\tau$. 

For $\rho\leq\tau=\sigma$ we let $\pi_{\tau,\rho}:N_{\R}/\rho\to N_{\R}/\tau$ be the quotient map. We also write $\pi_{\tau}:=\pi_{\tau,\{0\}}:N_{\R}=N_{\R}/\{0\}\to N_{\R}/\tau$.

Let $w\in \R_{\geq0}\times(N_{\R}/\tau)$ and let $\{u_1,\ldots,u_r\}$ be a set of generators for $\Lattice\cap\tau^\vee$. For each neighborhood $U$ of $w\in \R_{\geq0}\times(N_{\R}/\tau)$ and each $m\in\R$ we let 
$$C(U,m):=\dcup_{\rho\leq\tau}\{(r,x)\in \R_{\geq0}\times N_{\R}/\rho\;:\;(\id_{\R_{\geq0}}\times\pi_{\tau,\rho})(r,x)\in U\ \&\  \angbra{u_j,x}<m\text{ for }u_j\in\rho^\perp\sdrop\tau^\perp\}.$$
The $C(U,m)$s form a neighborhood basis for the topology of $\R_{\geq0}\times N_{\R}(\tau)$ at $w$; see \cite[Remark 3.4]{Pay09}. We let 
\begin{align*}
\wt{C}(U,m)&:=C(U,m)\cap (\R_{\geq0}\times N_{\R})\\
&=\{(r,x)\in \R_{\geq0}\times N_{\R}\;:\;(\id_{\R_{\geq0}}\times\pi_{\tau})(r,x)\in U\ \&\  \angbra{u_j,x}<m\text{ for }u_j\notin\tau^\perp\}.
\end{align*}

We now prove a series of claims.

Claim 1: For each $i=0,1,\ldots,k$ there is a $\what{w}_i\in L$ such that $(\id_{\R_{\geq0}}\times\pi_\tau)(\what{w}_i)=w_i$. Otherwise, we can pick an $i$ for which this fails. Since $\id_{\R_{\geq0}}\times\pi_\tau$ is linear, $(\id_{\R_{\geq0}}\times\pi_\tau)(L)$ is a polyhedral cone in $\R_{\geq0}\times N_{\R}/\tau$. Thus $(\id_{\R_{\geq0}}\times\pi_\tau)(L)$ is a closed set of $\R_{\geq0}\times N_{\R}/\tau$ and, by hypothesis, $w_i\notin(\id_{\R_{\geq0}}\times\pi_\tau)(L)$. So there is a neighborhood $U$ of $w_i$ in $\R_{\geq0}\times (N_{\R}/\tau)$ that does not meet $\tau$. But then, for any $m$, $C(U,m)$ is a neighborhood of $w_i$ in $\R_{\geq0}\times N_{\R}(\tau)$ that does not meet $L$, contradicting $w_i\in\cl_{\R_{\geq0}\times N_{\R}(\tau)}(L)$. This proves Claim 1.

We now let $L_0:=\{v\in N_{\R}\;:\;(0,v)\in L\}$.

Claim 2: $L_0^\vee\cap(-\tau^\vee\sdrop\tau^\perp)=\emptyset$. Towards seeing this, note that for any $u\in\Lattice_\R$ we have $u\in L_0^\vee$ if and only if $L_0\subseteq u^\vee$ which, in turn, happens if and only if there is an $a\in\R$ such that $L\subseteq (a,u)^\vee$. 

Now fix $u\in-\tau^\vee\sdrop\tau^\perp$; we want to show that $u\notin L_0^\vee$. For any $a\in\R$ we set $\calH(a,u):=\{(r,x)\in\R_{\geq0}\times N_{\R}(\tau)\;:\;-ar+\angbra{-u,x}\geq0\},$ which is a well-defined closed subset of $\R_{\geq0}\times N_{\R}(\tau)$ because $-u$ is a continuous function on $N_{\R}(\tau)$. Since $\calH(a,u)\cap\big(\R_{\geq0}\times N_{\R}\big)=(a,u)^\vee$, 
we have that $u\in L_0^\vee$ if and only if there is an $a\in\R$ such that $L\subseteq\calH(a,u)$.

Fix any $i=0,1,\ldots,k$ and $a\in\R$. Since $-u \in \tau^\vee \sdrop \tau^\perp$ and $w_i\in \R_{\geq0}\times N_{\R}/\tau$ we have $\angbra{-(a,u),w_i}=-\infty\not\geq 0_{\R}$, so $w_i\notin \calH(a,u)$. Because $w_i\in\cl_{\R_{\geq0}\times N_{\R}(\tau)}(L)$ and $\calH(a,u)$ is closed in $\R_{\geq0}\times N_{\R}(\tau)$, this implies that $L\not\subseteq\calH(a,u)$. Since this is true for all $a$, we conclude that $u\notin L_0^\vee$. This completes the proof of Claim 2.

Claim 3: $L_0\cap\relint(\tau)\neq\emptyset$. If not then $L_0\cap\tau$ is contained in a proper face of $\tau$. 
That is, there is some $u\in\tau^\vee\sdrop\tau^\perp$ such that $L_0\cap\tau\subseteq u^\perp$. In particular, $L_0\cap\tau\subseteq -u^\vee$, so $-u\in(L_0\cap\tau)^\vee=L_0^\vee+\tau^\vee$. Thus, there exists $\wt{u}\in\tau^\vee$ such that $-u\in L_0^\vee+\wt{u}$, so $-(u+\wt{u})\in L_0^\vee$. But $u\in\tau^\vee\sdrop\tau^\perp$ and $\wt{u}\in\tau^\vee$, so $-(u+\wt{u})\in-\tau^\vee\sdrop\tau^\perp$, contradicting Claim 2. This proves Claim 3.

By Claim 1 we have $\what{w}_0,\what{w}_1,\ldots,\what{w}_k$ such that $(\id_{\R_{\geq0}}\times\pi_\tau)(\what{w}_i)=w_i$ for all $i$. By Claim 3, we can pick $v'\in L_0\cap\relint(\tau)$. Let $v=(0,v')\in L$. 
Since $v'\in\tau$, we have $(\id_{\R_{\geq0}}\times\pi_\tau)(v)=0$, so for all $\calN\in\R$, $(\id_{\R_{\geq0}}\times\pi_\tau)(\what{w}_i+\calN v)=w_i$. Because $v'\in\relint(\tau)$, for any $u\in\tau^\vee\sdrop\tau^\perp$ we have $\angbra{u,v'}<0$. 

For each neighborhood $U$ of $w_i$ in $\R_{\geq0}\times(N_{\R}/\tau)$ and $m\in\R$, the previous two sentences together with the definition of $\wt{C}(U,m)$ show us that $\what{w}_i+\calN v\in\wt{C}(U,m)$ for all $\calN\gg0$. That is, $\dlim_{\calN\to\infty}\what{w}_i+\calN v=w_i$.
\end{proof}

We deduce from this a corollary for closures of polyhedra in (usual) tropical toric varieties.

\begin{coro}\label{coro: toric_Application}
Let $\Sigma$ be a fan in $N_{\R}$ and let $w\in N_{\R}(\Sigma)\sdrop N_{\R}$. If $\mathfrak{L}\subseteq N_{\R}$ is a polyhedron such that $w\in\cl_{N_{\R}(\Sigma)}(\mathfrak{L})$, then there are $\what{w}\in \mathfrak{L}$ and $v\in\operatorname{rec}(\mathfrak{L})$ such that $\dlim_{\largeNum\to\infty}\what{w}+\largeNum v=w$.
\end{coro}

\begin{proof}

    Fix any $\sigma\in\Sigma$ with $w\in N_{\R}(\sigma)$. Then $w\in\cl_{N_{\R}(\sigma)}\mathfrak{L}$ so we can work entirely in $N_{\R}(\sigma)$.

    Let 
    $L\subseteq \R_{\geq0}\times N_{\R}$ be the closed cone over $\mathfrak{L}$. Applying Lemma~\ref{lemma:PolyhedronClosureMeansEndOfRay} to $L$ and $(1,w)$ we get that there are $(a,\what{w}),(b,v)\in L$ such that $\dlim_{\largeNum\to\infty}(a,\what{w})+\largeNum(b,v)=(1,w)$. Thus $\dlim_{\largeNum\to\infty}\what{w}+\largeNum v=w$ and, because $\dlim_{\largeNum\to\infty}a+\largeNum b=1$, we must have $a=1$ and $b=0$. Thus $\what{w}\in \mathfrak{L}$ and $v\in\operatorname{rec}(\mathfrak{L})$.
\end{proof}

\newcommand{\Bd}{\operatorname{Bd}}

\begin{remark}\label{rmk: closures of polyhedra}
While Corollary~\ref{coro: toric_Application} is not particularly surprising, we do not see any very quick proof of it. For example, one approach to proving this would be to reduce to the case where $\mathfrak{L}$ is full-dimensional and then consider cases as to whether $w$ is in the interior of $\cl_{N_{\R}(\Sigma)}\mathfrak{L}$ or not. If $w$ is in the interior then there is a basic open neighborhood of $w$ (see \cite[Remark 3.4]{Pay09} for the relevant neighborhood basis) that is completely contained in $\cl_{N_{\R}(\Sigma)}\mathfrak{L}$; using the structure of such a neighborhood it would then not be difficult to show that the relevant $\what{w}$ and $v$ exist inside this neighborhood. If $w$ is not in the interior, i.e., $w$ is in the boundary of $\cl_{N_{\R}(\Sigma)}\mathfrak{L}$, then one would hope to show that $w$ is in the closure of a proper face of $\mathfrak{L}$ and then use induction on dimension. However, in order to do this, one would need to show that $\Bd_{N_{\R}(\Sigma)} \left( \cl_{N_{\R}(\Sigma)} \mathfrak{L} \right) \subseteq \cl_{N_{\R}(\Sigma)} \left( \Bd_{N_{\R}}\mathfrak{L} \right)$ and, while intuitively clear, this does not seem to have a straightforward proof.

Ultimately, our proof takes a different, case-free, approach which allows us to prove slightly more; see Lemma~\ref{lemma:PolyhedronClosureMeansEndOfRay}. The proof takes about one page. %We bring this result to the attention of the community to save people time in not having to reproduce it or similar results from scratch.  
\end{remark}

\begin{lemma}\label{lemma:VandW}
Let $v,\what{w}\in\R_{\geq0}\times N_{\R}$ and $w\in\R_{\geq0}\times N_{\R}(\sigma)$ and suppose that $\dlim_{\largeNum\to\infty}\what{w}+\largeNum v=w$. Then, for any $u\in\R\times\Mon$,
\begin{enumerate}
\item\label{lemmaPart:infinite} $\angbra{u,w}=-\infty$ if and only if $\angbra{u,v}<0$.

\item\label{lemmaPart:finite} $\angbra{u,w}\neq-\infty$ if and only if $\angbra{u,v}=0$. If these occur, then $\angbra{u,\what{w}}=\angbra{u,w}$.
\end{enumerate}
\end{lemma}
\begin{proof}

For any $u\in\R\times \Mon$ we have 
$$\angbra{u,w}=\dlim_{\calN\to\infty}\left< u,\what{w}+\calN v\right> = \left< u, \what{w}\right> + \dlim_{\calN\to\infty}\calN \left< u,v\right>.$$
Since $\angbra{u,w}$ is either a finite number or $-\infty$, we must have $\angbra{u,v}=0$ or $\angbra{u,v}<0$. If $\angbra{u,v}=0$ then 
$$\angbra{u,w}=\left< u, \what{w}\right> + \dlim_{\calN\to\infty}\calN \left< u,v\right>=\angbra{u,\what{w}},$$
which is a finite number. If $\angbra{u,v}<0$ then 
$$\angbra{u,w}=\dlim_{\calN\to\infty}\left< u,\what{w}+\calN v\right> = \left< u, \what{w}\right> + \dlim_{\calN\to\infty}\calN \left< u,v\right>=-\infty.$$\par\nopagebreak\vspace{-1.5\baselineskip}\mbox{}
\end{proof}

%\orange{The following lemma is fairly straightforward.}
The following lemma is intuitively clear to experts, but we could not find the exact statement in the literature.

\newcommand{\prefan}[1]{\wt{\Pi}_{#1}}

\begin{lemma}\label{lemma:SetValuedFunctionProperties}
Let $\ph:\frakV\to\frakW$ be a linear map of (finite dimensional) real vector spaces. Let $\Sigma$ be a fan in $\frakV$ and define a set-valued function on $\frakW$ by 
\begin{align*}
\Phi(w)&=\left\{\rho\in\Sigma \;:\; \rho\cap\ph^{-1}(w) \neq\emptyset \right\}\\
&=\left\{\rho\in\Sigma \;:\; w\in\ph(\rho) \right\}.
\end{align*}
Then $\Phi$ has the following two properties.
\begin{enumerate}
\item There is a complete fan $\Pi$ in $\frakW$ such that $\Phi$ is constant on the relative interior of each cone in $\Pi$. In particular, the map $\Phi$ is piecewise constant.
\item For every $w\in\frakW$, there is a neighborhood $\frakN$ of $w$ such that, for all $\xi\in \frakN$, $\Phi(\xi)\subseteq\Phi(w)$. (This is a semicontinuity-type property\footnote{In fact, this property says that $\Phi$ is \emph{upper hemicontinuous} \cite[Definition 17.2]{AB} when the topology on $\Sigma$ is, for example, the discrete topology.}.)
\end{enumerate}
\end{lemma}
\begin{proof}
Fix $\rho\in\Sigma$ and pick a fan $\prefan{\rho}$ that subdivides the cone $\ph(\rho)$ in $\frakW$; if $\ph(\rho)$ is strictly convex, we can take $\prefan{\rho}$ to be the fan of faces of $\ph(\rho)$. Let $\Pi_\rho$ be any completion of $\prefan{\rho}$ (see for the existence of completions of rational fans \cite[page 18]{Oda88}, \cite{EI06}, and \cite{Roh11}). For any $w\in\frakW$ there is a unique $\tau\in\Pi_\rho$ with $w\in\relint\tau$ and $\rho\cap\ph^{-1}(w)\neq\emptyset$ if and only if $\tau\in\prefan{\rho}$.

Now let $\Pi$ be the common refinement of the $\Pi_\rho$ for $\rho\in\Sigma$. That is, 
$$\Pi=\left\{\dcap_{\rho\in\Sigma}\tau_\rho\;:\; \tau_\rho\in\Pi_\rho\text{ for all }\rho\in\Sigma\right\}.$$
Since each $\Pi_\rho$ is complete, $\Pi$ is complete. For each $\lambda\in\Pi$ and $\rho\in\Sigma$ there is a unique smallest cone of $\Pi_\rho$ that contains $\lambda$; we call this cone $\tau_\rho(\lambda)$. Note that $\relint\lambda\subseteq\relint\tau_\rho(\lambda)$. So if $w\in\relint\lambda$ then $\rho\cap\ph^{-1}(w)\neq\emptyset$ if and only if $\tau_\rho(\lambda)\in\prefan{\rho}$.

Now fix $\lambda\in\Pi$ and suppose $w,v\in\relint\lambda$. Then 
\begin{align*}
\Phi(w)&=\left\{\rho\in\Sigma\;:\;\rho\cap\ph^{-1}(w)\neq\emptyset\right\}\\
&=\left\{\rho\in\Sigma\;:\;\tau_\rho(\lambda)\in\prefan{\rho}\right\}\\
&=\left\{\rho\in\Sigma\;:\;\rho\cap\ph^{-1}(v)\neq\emptyset\right\}\\
&=\Phi(v),
\end{align*}
proving (1).

Fix $w\in\frakW$ and let $\frakN=\dcup_{\substack{\lambda\in\Pi\\\lambda\ni w}}\left(\relint\lambda\right)$ be the open star  of $w$ in $\Pi$. Since $\Pi$ is complete, $\frakN$ is open. Let $\theta$ be the smallest cone of $\Pi$ with $w\in\theta$, so $w\in\relint\theta$.

Now fix $\xi\in\frakN$; we show that $\Phi(\xi)\subseteq\Phi(w)$. So there is a $\lambda\in\Pi$ with $w\in\lambda$ and $\xi\in\relint\lambda$. Since $\theta$ is the smallest cone of $\Pi$ containing $w$, $\theta$ must be a face of $\lambda$. In particular, for any $\rho$ we have that $\tau_\rho(\lambda)\in\Pi_\rho$ contains $\lambda$ and therefore contains $\theta$. Since $\tau_\rho(\theta)$ is the smallest cone of $\Pi_\rho$ that contains $\theta$, $\tau_\rho(\theta)$ must be a face of $\tau_\rho(\lambda)$. 

Now suppose $\rho\in\Phi(\xi)$; we want to show $\rho\in\Phi(w)$. By definition of $\Phi$, $\rho\cap\ph^{-1}(\xi)\neq\emptyset$, which we know is equivalent to $\tau_\rho(\lambda)\in\prefan{\rho}$. Since $\tau_\rho(\theta)$ is a face of $\tau_\rho(\lambda)$ and $\prefan{\rho}$ is a fan, $\tau_\rho(\theta)\in\prefan{\rho}$. Equivalently, $\rho\cap\ph^{-1}(w)\neq\emptyset$, so $\rho\in\Phi(w)$.
\end{proof}

\begin{remark}\label{rmk:HalfSpacesToo}
The above proof also works when  $\ph:\frakV\to\frakW$ is a linear map of (finite dimensional) half-spaces.
\end{remark}

In the proof of our next lemma, it will be convenient to have the following definition.

\begin{defi}
Let $\ph:X\to Y$ be an affine map of real vector spaces and let $\Pi$ be a polyhedral complex in $Y$. The \emph{pullback of $\Pi$ along $\ph$} is the polyhedral complex in $X$ given by 
$$\ph^*(\Pi):=\{\ph^{-1}(\mathfrak{Q})\;:\;\mathfrak{Q}\in\Pi\}.$$
\end{defi}

\begin{lemma}\label{lemma:SimplexOfCoefficients}
Let $\Pi$ be a polyhedral complex whose support is a convex cone $\lambda=\supp(\Pi)$. Suppose $w_0,\ldots,w_k\in\lambda$. Then there are a full-dimensional simplex $\simon$ in $\R^k$ of the form 
$$\simon=\{(\delta_1,\ldots,\delta_k)\;:\; 0\leq\delta_1\leq a_{1} \text{ and } 0\leq\delta_i\leq a_i\delta_{i-1} \text{ for } i=2,\ldots,k\}$$
for some positive constants $a_1,\ldots,a_k$ and a polyhedron $\paula\in\Pi$ such that, for all $(\delta_1,\ldots,\delta_k)\in\simon$, $w_0+\dsum_{i=1}^k\delta_iw_i$ is contained in $\paula$ and, if $\delta_k\neq0$, $w_0+\dsum_{i=1}^k\delta_iw_i \in \relint\paula$.
\end{lemma}\begin{proof}
We prove this statement by induction on $k\geq1$.

For the base case, let $k=1$. Let $\ph$ be the map from $\R$ to the ambient space of $\Pi$ given by $\ph(t)=w_0+tw_1$. Since $w_0,w_1\in\lambda$ and $\lambda=\supp(\Pi)$ is a cone, the support of $\ph^*(\Pi)$ contains $\R_{\geq0}$. Now consider the polyhedra of $\ph^*(\Pi)$ that contain $0_{\R}$. Since the support of $\ph^*(\Pi)$ contains $\R_{\geq0}$, there must be one such polyhedron that also contains positive numbers; call this polyhedron $I$. Then $I$ is an interval with one endpoint a positive number that we call $\alpha$. Let $a_1=\alpha/2$. So $[0,a_1]$ is contained in $I$ and $(0,a_1]$ is contained in $\relint I$.

Let $\paula$ be the smallest polyhedron of $\Pi$ with $\ph^{-1}(\paula)=I$ and let $\simon=[0,a_1]$. Certainly, if $\delta_1\in\simon$ then $w_0+\delta_1w_1=\ph(\delta_1)\in\paula$. If $\delta_1\neq0$, then $I$ is the smallest polyhedron of $\ph^*(\Pi)$ containing $\delta_1$, so any polyhedron $\frakQ\in\Pi$ that contains $\ph(\delta_1)$ must have $\ph^{-1}(\frakQ)\supset I$; since $I$ is full-dimensional in $\R$, this means $\ph^{-1}(\frakQ)=I$. So, because $\paula$ is the smallest polyhedron of $\Pi$ with $\ph^{-1}(\paula)=I$, $\paula$ is the smallest polyhedron of $\Pi$ containing $\ph(\delta_1)$. In particular, $\ph(\delta_1)$ is not in any proper face of $\paula$, so $w_0+\delta_1w_1=\ph(\delta_1)\in\relint\paula$. This completes the proof of the base case.

Now suppose the result is true for $k$ and that we have $w_0,\ldots,w_k,w_{k+1}\in\lambda$. By the inductive hypothesis there are a full-dimensional simplex $\rachel$ in $\R^k$ of the form 
$$\rachel=\{(\delta_1,\ldots,\delta_k)\;:\; 0\leq\delta_1\leq a_{1} \text{ and } 0\leq\delta_i\leq a_i\delta_{i-1} \text{ for } i=2,\ldots,k\}$$
for some positive constants $a_1,\ldots,a_k$ and a polyhedron $\olive\in\Pi$ such that, for all $(\delta_1,\ldots,\delta_k)\in\rachel$, $w_0+\dsum_{i=1}^k\delta_iw_i$ is contained in $\olive$ and, if $\delta_k\neq0$, $w_0+\dsum_{i=1}^k\delta_iw_i \in \relint\olive$. Let $b_0=1$ and for $i=1,\ldots,k$ let $b_i=a_ib_{i-1}$, i.e., $b_i=\dprod_{j=1}^ia_j$. For $i=0,\ldots,k$ let $\zeta_i=\dsum_{j=0}^i b_jw_j$. Then $\zeta_i\in\olive$ for $i=0,\ldots,k$ and $\zeta_k\in\relint\olive$.

Applying the base case with $\zeta_k$ and $w_{k+1}$, we find that there is a positive real number $b_{k+1}$ and a $\paula\in\Pi$ such that, if $0\leq t\leq b_{k+1}$ then $\zeta_k+tw_{k+1}\in\paula$ and if $t\neq0$ then $\zeta_k+tw_{k+1}\in\relint\paula$. We let $\zeta_{k+1}:=\zeta_k+b_{k+1}w_{k+1}$, so $\zeta_{k+1}\in\relint\paula$. Now $\zeta_k=\zeta_k+0w_k\in\paula\cap\olive$ and so, because $\olive$ is the smallest polyhedron in $\Pi$ containing $\zeta_k$, $\olive\subseteq\paula\cap\olive$, i.e., $\olive\subseteq\paula$. Thus $\zeta_i\in\paula$ for $i=0,\ldots,k+1$.

Let $a_{k+1}:=\dfrac{b_{k+1}}{b_k}$ and let 
$$\simon=\{(\delta_1,\ldots,\delta_{k+1})\in\R^{k+1}\;:\; 0\leq\delta_1\leq a_{1} \text{ and } 0\leq\delta_i\leq a_i\delta_{i-1} \text{ for } i=2,\ldots,k+1\}.$$
Fix $(\delta_1,\ldots,\delta_{k+1})\in\simon$; we will show that $v:=w_0+\dsum_{i=1}^{k+1}\delta_iw_i\in\paula$. Let $r_0:=1-\dfrac{\delta_1}{b_1}$, $r_i:=\dfrac{\delta_i}{b_i}-\dfrac{\delta_{i+1}}{b_{i+1}}$ for $i=1,\ldots,k$, and $r_{k+1}=\dfrac{\delta_{k+1}}{b_{k+1}}$. Note that each $r_i$ is non-negative. For $r_{k+1}$ this is clear and for $r_0$ this follows from $\dfrac{\delta_1}{b_1}=\dfrac{\delta_1}{a_1}\leq 1$. For $i=1,\ldots,k$ we have $\dfrac{\delta_{i+1}}{a_{i+1}}\leq\delta_i$ and dividing through by $b_{i}$ gives $\dfrac{\delta_{i+1}}{b_{i+1}}=\dfrac{\delta_{i+1}}{a_{i+1}b_i}\leq\dfrac{\delta_i}{b_i}$, so $r_i\geq0$. 

Note that $\dsum_{i=0}^{k+1}r_i=1$ and, for $0<j\leq k+1$, $\dsum_{i=j}^{k+1}r_i=\dfrac{\delta_j}{b_j}$. So we have
\begin{align*}
\dsum_{i=0}^{k+1}r_i\zeta_i&=\dsum_{i=0}^{k+1} \left(r_i\dsum_{j=0}^i b_jw_j\right)
=\dsum_{0\leq j\leq i\leq k+1}r_ib_jw_j\\
&=\left(\dsum_{i=0}^{k+1}r_i\right)b_0w_0+\dsum_{j=1}^{k+1}\left(\dsum_{i=j}^{k+1}r_i\right)b_jw_j\\
&=w_0+\dsum_{j=1}^{k+1}\dfrac{\delta_j}{b_j}b_jw_j\\
&=w_0+\dsum_{j=1}^{k+1}\delta_jw_j=v.
\end{align*}
Thus $v$ is a convex combination of $\zeta_0,\ldots,\zeta_{k+1}\in\paula$ and, since $\paula$ is convex, $v\in\paula$, as claimed.

Continuing with $(\delta_1,\ldots,\delta_{k+1})\in\simon$, we now suppose $\delta_{k+1}>0$; we will show that $v\in\relint\paula$. Consider a nonzero dual vector $u$ and a real number $c$ with $\paula\subseteq\{\xi:\angbra{u, \xi}\leq c\}$; we will show that $v$ is not contained in the proper face $\paula\cap\{\xi:\angbra{u, \xi}=c\}$ of $\paula$. Since $\zeta_{k+1}\in\relint\paula$, $\angbra{u,\zeta_{k+1}}<c$. So 
$$\angbra{u,v}=\dsum_{i=0}^{k+1}r_i\angbra{u,\zeta_i} \leq \dsum_{i=0}^{k}r_ic + r_{k+1}\angbra{u,\zeta_{k+1}}<\dsum_{i=0}^{k+1}r_ic=c.$$
Thus $v$ is not contained in any proper face of $\paula$, so $v\in\relint\paula$.
\end{proof}

\begin{lemma}\label{lemma:GeneralClosureImpliesAConeCondition}
Let $X^\circ$ be the support of a fan in $\R_{\geq0}\times N_{\R}$ and let $X=\cl_{\R_{\geq0}\times N_{\R}(\sigma)}X^\circ$.
Let $\tau$ be a face of $\sigma$, let $\carl$ be a cone contained in $X\cap[\R_{\geq0}\times(N_{\R}/\tau)]$, and let $w_0,w_1,\ldots,w_k\in\carl$.
Then there are $v,\what{w}_0,\what{w}_1,\ldots,\what{w}_k\in\R_{\geq0}\times N_{\R}$ and an open set $U\subseteq(\R_{>0})^{k+1}$ such that
\begin{itemize}
\item $\dlim_{\largeNum\to\infty}\what{w}_i+\largeNum v=w_i$ for $i=0,1,\ldots,k$,
\item whenever $1\gg\eps_0\gg\eps_1\gg\cdots\gg\eps_k$, we have $(\eps_0,\eps_1,\ldots,\eps_k)\in U$, and
\item if $(\eps_0,\eps_1,\ldots,\eps_k)\in U$ then $v+\dsum_{i=0}^k\eps_i\what{w}_i\in X^\circ$.
\end{itemize}
\end{lemma}

\begin{proof}
Fix a fan structure $\Sigma$ on $X^\circ$ and let $\pi_\tau:N_{\R}\to N_{\R}/\tau$ be the quotient map. We define a set-valued function $\Phi$ on $\R_{\geq0}\times(N_{\R}/\tau)$ with values that are subsets of $\Sigma$ as follows. For any $w\in \R_{\geq0}\times(N_{\R}/\tau)$, 
\begin{align*}
\Phi(w)&=\left\{\rho\in\Sigma \;:\; \rho\cap(\id_{\R_{\geq0}}\times\pi_\tau)^{-1}(w) \neq\emptyset \right\}\\
&=\left\{\rho\in\Sigma \;:\; \text{for some } \what{w}\in\rho,\ (\id_{\R_{\geq0}}\times\pi_\tau)(\what{w})=w \right\}.
\end{align*}

Lemma~\ref{lemma:SetValuedFunctionProperties} and Remark~\ref{rmk:HalfSpacesToo} now tell us that 
\begin{enumerate}
\item there is a complete fan $\Pi$ in $\R_{\geq0}\times(N_{\R}/\tau)$ such that $\Phi$ is constant on the relative interior of each cone in $\Pi$ and
\item for every $w\in \R_{\geq0}\times(N_{\R}/\tau)$, there is a neighborhood $\frakN$ of $w$ such that, for all $\xi\in \frakN$, $\Phi(\xi)\subseteq\Phi(w)$.
\end{enumerate}
 
Since $\supp(\Pi)=\R_{\geq0}\times(N_{\R}/\tau)$ is a convex cone, Lemma~\ref{lemma:SimplexOfCoefficients} tells us that there is a full-dimensional simplex $\simon$ in $\R^{k}$ of the form 
$$\simon=\{(\delta_1,\ldots,\delta_k)\;:\; 0\leq\delta_1\leq a_{1} \text{ and } 0\leq\delta_i\leq a_i\delta_{i-1} \text{ for } i=2,\ldots,k\}$$
for some positive constants $a_1,\ldots,a_k$ and a cone $\paula\in\Pi$ such that, for all $(\delta_1,\ldots,\delta_k)\in\simon$, $w_0+\dsum_{i=1}^k\delta_iw_i\in\paula$ and, if $(\delta_1,\ldots,\delta_k)\in\relint\simon$ then $w_0+\dsum_{i=1}^k\delta_iw_i\in\relint\paula$.  Letting $\Opal=\left\{w_0+\dsum_{i=1}^k\delta_iw_i\;:\;(\delta_1,\ldots,\delta_k)\in\relint\simon\right\}$, we have that the value $\Phi(\omega)$ is independent of $\omega\in\Opal$; we call this value $\tara$. 

For $i=0,\ldots,k$ we let $b_i=\dprod_{j=1}^i a_j$; we have $b_0=1$. Now for $i=0,\ldots,k$ we let $\valerie_i=\dsum_{j=0}^i b_jw_j$.
Note that each $\valerie_i$ is in the closure of $\Opal$, i.e., every neighborhood of $\valerie_i$ meets 
$\Opal$.  So (2) tells us that $\tara\subseteq\Phi(\valerie_i)$.

Fix $\omega\in\Opal$. Since $w_0,w_1,\ldots,w_k\in\carl$, we also have 
$$\omega 
=w_0+\dsum_{i=1}^k\delta_iw_i\in\carl\subseteq X
=\cl_{\R_{\geq0}\times N_{\R}(\sigma)}(X^\circ)
=\cl_{\R_{\geq0}\times N_{\R}(\sigma)}(\supp(\Sigma))
=\dcup_{L\in\Sigma}\cl_{\R_{\geq0}\times N_{\R}(\sigma)}(L), 
$$
so we can fix an $L\in\Sigma$ with $\omega\in\cl_{\R_{\geq0}\times N_{\R}(\sigma)}(L)$. Lemma~\ref{lemma:PolyhedronClosureMeansEndOfRay} now tells us that $L\in\Phi(\omega)=\tara$ and $L\cap[\{0_{\R}\}\times(\relint\tau)]\neq\emptyset$. Fix $v\in L\cap[\{0_{\R}\}\times(\relint\tau)]$. 

For each $i=0,\ldots,k$ we have $L\in\tara\subseteq\Phi(\valerie_i)$ so there is a $\what{\valerie}_i\in L$ with $(\id_{\R_{\geq0}}\times \pi_\tau) \left(\what{\valerie}_i\right)=\valerie_i$. 
Let $\what{w}_0=\what{\valerie}_0$ and, for $i=1,\ldots,k$, $\what{w}_i=\dfrac{1}{b_i}\left(\what{\valerie}_i-\what{\valerie}_{i-1}\right)$. Since $(\id_{\R_{\geq0}}\times \pi_\tau)$ is linear, we have
$$
(\id_{\R_{\geq0}}\times \pi_\tau)\left(\what{w}_i\right)
=\dfrac{1}{b_i}\left(
(\id_{\R_{\geq0}}\times \pi_\tau)\left(\what{\valerie}_i\right) - (\id_{\R_{\geq0}}\times \pi_\tau)\left(\what{\valerie}_{i-1}\right)
\right)
=\dfrac{1}{b_i}\left(\valerie_i-\valerie_{i-1}\right)=w_i
$$
for $i=1,\ldots,k$ and $(\id_{\R_{\geq0}}\times \pi_\tau)\left(\what{w}_0\right)=(\id_{\R_{\geq0}}\times \pi_\tau)\left(\what{\valerie}_0\right)=\valerie_0=w_0$.
Now Lemma~\ref{lemma:PolyhedronClosureMeansEndOfRay} tells us that $\dlim_{\largeNum\to\infty}\what{w}_i+\largeNum v=w_i$ for $i=0,\ldots,k$.

Let 
$$U=\left\{(\eps_0,\eps_1,\ldots,\eps_k)\in\R^k\;:\;0<\eps_0<1 \text{ and } 0<\eps_i<a_i\eps_{i-1} \text{ for } i=1,\ldots,k\right\}.$$
Certainly, if $1\gg\eps_0\gg\eps_1\gg\cdots\gg\eps_k$ then $(\eps_0,\eps_1,\ldots,\eps_k)\in U$.

Say $(\eps_0,\eps_1,\ldots,\eps_k)\in U$. We claim that $\dfrac{\eps_i}{b_i}<\dfrac{\eps_{i-1}}{b_{i-1}}$ for all $i=1,\ldots,k$. Recall that $b_0=1$ and $b_i=a_ib_{i-1}$ for $i=1,\ldots,k$. So, for $i=1$, we have $\dfrac{\eps_1}{b_1}=\dfrac{\eps_1}{a_1}<\eps_0=\dfrac{\eps_0}{b_0}$. For $i>1$ we have $\dfrac{\eps_i}{a_i}<\eps_{i-1}$ so multiplying by $\dfrac{1}{b_{i-1}}$ we get  
$\dfrac{\eps_i}{b_i}=\dfrac{\eps_i}{a_ib_{i-1}}<\dfrac{\eps_{i-1}}{b_{i-1}}$.

Thus, the numbers $\dfrac{\eps_{i}}{b_{i}}-\dfrac{\eps_{i+1}}{b_{i+1}}$ for $i=0,\ldots,k-1$ and $\dfrac{\eps_k}{a_k}$ are all positive. We let $\what{\valerie}_{-1}=0$ so that $\what{w_0}=\dfrac{1}{b_0}\left(\what{\valerie}_0-\what{\valerie}_{-1}\right)$. Now, if $(\eps_0,\eps_1,\ldots,\eps_k)\in U$ then
\begin{align*}
v+\dsum_{i=0}^k\eps_i\what{w}_i&=v+\dsum_{i=0}^k\dfrac{\eps_i}{b_i}\left(\what{\valerie}_i-\what{\valerie}_{i-1}\right)\\
&=v+\dsum_{i=0}^k\dfrac{\eps_i}{b_i}\what{\valerie}_i-\dsum_{i=0}^k\dfrac{\eps_i}{b_i}\what{\valerie}_{i-1}\\
&=v+\dsum_{i=0}^k\dfrac{\eps_i}{b_i}\what{\valerie}_i-\dsum_{i=-1}^{k-1}\dfrac{\eps_{i+1}}{b_{i+1}}\what{\valerie}_{i}\\
&=v+\dsum_{i=0}^{k-1}\dfrac{\eps_i}{b_i}\what{\valerie}_i +\dfrac{\eps_k}{b_k}\what{\valerie}_k - \dfrac{\eps_{0}}{b_{0}}\what{\valerie}_{-1} -\dsum_{i=0}^{k-1}\dfrac{\eps_{i+1}}{b_{i+1}}\what{\valerie}_{i}\\
&=v+\dsum_{i=0}^{k-1}\left(\dfrac{\eps_i}{b_i}-\dfrac{\eps_{i+1}}{b_{i+1}}\right)\what{\valerie}_i +\dfrac{\eps_k}{b_k}\what{\valerie}_k-0\\
&\in L\subseteq X^\circ
\end{align*}
where the last line holds because $L$ is a cone and $v,\what{\valerie}_0,\what{\valerie}_1,\ldots,\what{\valerie}_k\in L$.
\end{proof}

\begin{proof}[Proof of Theorem~\ref{thm:resolving-primes}]
Because the prime congruences of $\base[\Mon]/E$ are in bijection with the prime congruences of $\base[\Mon]$ containing $E$, it suffices to show that, if $P$ is a prime congruence of $\base[\Mon]$ with $E\subseteq P$ then there is a prime congruence $Q$ on $\base[\Mon]$ such that $E\subseteq Q\subseteq P$ and $Q$ has trivial ideal-kernel. Let $\tau=\tau_P$.

Since $E\subseteq P$, Theorem~\ref{thm:PrimeContainsCongruence} tells us that we can write $P=P_{\calc_\bullet}$ for a flag of cones $\calc_\bullet=(\calc_{-1}\leq\calc_{0}\leq\cdots\leq\calc_k)$ contained in $\R_{\geq0}\times (N_{\R}/\tau)$ such that, for all $i$, $\calc_i\subseteq \wt{V}(E)$. For $i=0,1,\ldots,k$, pick $w_i\in\calc_i\sdrop\calc_{i-1}$. So $P$ is given by the matrix 
$\begin{pmatrix}w_0\\w_1\\\vdots\\w_k\end{pmatrix}$. 
 
Applying Lemma~\ref{lemma:GeneralClosureImpliesAConeCondition} with $X^\circ=\wt{V}(E)\cap(\R_{\geq0}\times N_{\R})$, $X=\wt{V}(E)$ and $\carl=\calc_k$, we get that there are 
$v,\what{w}_0,\what{w}_1,\ldots,\what{w}_k\in\R_{\geq0}\times N_{\R}$ and an open set $U\subseteq(\R_{>0})^{k+1}$ such that
\begin{itemize}
\item $\dlim_{\largeNum\to\infty}\what{w}_i+\largeNum v=w_i$ for $i=0,1,\ldots,k$,
\item whenever $1\gg\eps_0\gg\eps_1\gg\cdots\gg\eps_k$, we have $(\eps_0,\eps_1,\ldots,\eps_k)\in U$, and
\item if $(\eps_0,\eps_1,\ldots,\eps_k)\in U$ then $v+\dsum_{i=0}^k\eps_i\what{w}_i\in \wt{V}(E)$.
\end{itemize}
Let $Q$ be the prime congruence on $\base[\Mon]$ defined by the matrix 
$\begin{pmatrix}v\\\what{w}_0\\\what{w}_1\\\cdots\\\what{w}_k\end{pmatrix}$. 
Since $v,\what{w}_0,\what{w}_1\ldots,\what{w}_k\in \R_{\geq0}\times N_{\R}$, $Q$ has trivial ideal-kernel. Fix $(f,g)\in E$; we will show that $(f,g)\in Q$. Given $(\eps_0,\eps_1,\ldots,\eps_k)\in U$, let $\omega=v+\dsum_{i=0}^k \eps_i\what{w}_i$; since $\omega \in \wt{V}(E)$ and $(f,g)\in E$, we have $\wt{f}(\omega)=\wt{g}(\omega)$.

Since $Q$ is prime, $(f, g) \in Q$ if and only if $(\whin_{Q}(f), \whin_{Q}(g))\in Q.$ Recall that 
\begin{align*}
\whin_Q(f)&=\whin_{\what{w}_k}\left(
\whin_{\what{w}_{k-1}}\left(\cdots
\whin_{\what{w}_0}\left(
\whin_{v}(f)
\right)
\cdots\right)
\right) \\
\intertext{and}
\whin_Q(g)&=\whin_{\what{w}_k}\left(
\whin_{\what{w}_{k-1}}\left(\cdots
\whin_{\what{w}_0}\left(
\whin_{v}(g)
\right)
\cdots\right)
\right).
\end{align*}

Applying Corollary~\ref{coro:InitialFormsAddingManyVectors} with $f_1=f$, $f_2=g$, $\xi_j=\what{w}_{k-j}$ for $0\leq j\leq k$ and $\xi_{k+1}=v$, we get that there is a full-dimensional polyhedron $\frakQ$ in $\R^{k+1}$ such that, if $(\largeNum_{1},\largeNum_2,\ldots,\largeNum_{k},\largeNum_{k+1})$ is in the interior of $\frakQ$ then we have 
$$\whin_{\what{w}_k}\left(\whin_{\what{w}_{k-1}}\left(\cdots\whin_{\what{w}_0}\left(\whin_{v}(f)\right)\cdots\right)\right)=\whin_{\what{w}_k+\largeNum_{1} \what{w}_{k-1}+\cdots+\largeNum_{k} \what{w}_0+\largeNum_{k+1} v}(f)$$ 
and 
$$\whin_{\what{w}_k}\left(\whin_{\what{w}_{k-1}}\left(\cdots\whin_{\what{w}_0}\left(\whin_{v}(g)\right)\cdots\right)\right)=\whin_{\what{w}_k+\largeNum_{1} \what{w}_{k-1}+\cdots+\largeNum_{k} \what{w}_0+\largeNum_{k+1} v}(g).$$ 
Moreover, if  $0\ll\largeNum_{1}\ll\largeNum_{2}\ll\cdots\ll\largeNum_{k}\ll\largeNum_{k+1}$, then $(\largeNum_{1},\largeNum_2,\ldots,\largeNum_{k},\largeNum_{k+1})$ is in the interior of $\frakQ$.

Note that, if $a\in\R$ is positive and $\zeta\in\R_{\geq0}\times N_{\R}(\sigma)$ then $\whin_{a\zeta}(f)=\whin_\zeta(f)$. So if we let $\eps_k=\dfrac{1}{\largeNum_{k+1}}$ and $\eps_j=\dfrac{\largeNum_{k-j}}{\largeNum_{k+1}}$ for $0\leq j\leq k-1$ then 
$$\whin_{\what{w}_k+\largeNum_{1} \what{w}_{k-1}+\cdots+\largeNum_{k} \what{w}_0+\largeNum_{k+1} v}(f)=
\whin_{\eps_k\what{w}_k+\eps_{k-1} \what{w}_{k-1}+\cdots+\eps_0 \what{w}_0+ v}(f)=\whin_{\omega}(f)$$
and
$$\whin_{\what{w}_k+\largeNum_{1} \what{w}_{k-1}+\cdots+\largeNum_{k} \what{w}_0+\largeNum_{k+1} v}(g)=
\whin_{\eps_k\what{w}_k+\eps_{k-1} \what{w}_{k-1}+\cdots+\eps_0 \what{w}_0+ v}(g)=\whin_{\omega}(g)$$

Also, the condition of $0\ll\largeNum_1\ll\largeNum_2\ll\cdots\ll\largeNum_{k}\ll\largeNum_{k+1}$ is equivalent to $1\gg\eps_0\gg\eps_1\gg\cdots\gg\eps_k$ and, because $(\largeNum_{1},\largeNum_{2},\ldots,\largeNum_{k},\largeNum_{k+1})\mapsto\left(\dfrac{\largeNum_{k}}{\largeNum_{k+1}},\dfrac{\largeNum_{k-1}}{\largeNum_{k+1}},\ldots,\dfrac{\largeNum_{1}}{\largeNum_{k+1}},\dfrac{1}{\largeNum_{k+1}}\right)$ is a homeomorphism from $(\R_{>0})^{k+1}$ to itself, there is an open set $U'$ of vectors $(\eps_0,\eps_1,\ldots,\eps_k)$ for which these equalities are valid. Moreover, since $1\gg\eps_0\gg\eps_1\gg\cdots\gg\eps_k$ implies both $(\eps_0,\eps_1,\ldots,\eps_k)\in U$ and $(\eps_0,\eps_1,\ldots,\eps_k)\in U'$, the open set $U\cap U'$ is nonempty.

In summary, we have seen that there is a nonempty open set $U\cap U'$ of vectors $(\eps_0,\eps_1,\ldots,\eps_k)$ such that, letting $\omega=v+\dsum_{i=0}^k \eps_i\what{w}_i$, we have $\wt{f}(\omega)=\wt{g}(\omega)$, $\whin_Q(f)=\whin_{\omega}(f)$, and $\whin_Q(g)=\whin_{\omega}(g)$.

Let $m_f$ be any term occurring in $\whin_Q(f)=\whin_{\omega}(f)$ and let $m_g$ be any term occurring in $\whin_Q(g)=\whin_{\omega}(g)$. Then $m_f$ and $m_g$ are $Q$-leading terms of $f$ and $g$, respectively, so $(f,g)\in Q$ if and only if $(m_f,m_g)\in Q$. Also, the maxima in $\wt{f}(\omega)$ and $\wt{g}(\omega)$ occur at $\wt{m_f}(\omega)$ and $\wt{m_g}(\omega)$, so $\wt{m_f}(\omega)=\wt{m_g}(\omega)$. Thus, for any $(\eps_0,\ldots,\eps_k)\in U\cap U'$, we have 
$$\wt{m_f}(v) +  \dsum_{i=0}^k \eps_i \wt{m_f}(\what{w}_i) =\wt{m_f}(\omega)=\wt{m_g}(\omega)
= \wt{m_g}(v) +  \dsum_{i=0}^k \eps_i \wt{m_g}(\what{w}_i).$$
Since the left and right hand sides are affine functions of the variables $(\eps_0,\ldots,\eps_k)$ that agree on $U\cap U'$, and because the affine span of any nonempty open set in $\R^{k+1}$ is all of $\R^{k+1}$, these affine functions must agree on all of $\R^{k+1}$. Thus the constants and coefficients on both sides must agree, i.e., $m_f(v)=m_g(v)$ and $m_f(\what{w}_i)=m_g(\what{w}_i)$ for $0\leq i\leq k$. Since $Q$ is the prime congruence given by the matrix 
$\begin{pmatrix}v\\\what{w}_0\\\what{w}_1\\\cdots\\\what{w}_k\end{pmatrix}$, this shows that $(m_f,m_g)\in Q$. Thus $(f,g)\in Q$, finishing the proof that $E\subseteq Q$.

It is now enough to prove $Q\subseteq P$. For this, it suffices to show that, for all $m_1,m_2\in\Terms{\base[\Mon]}$, if $m_1\leq_{Q}m_2$ then $m_1\leq_P m_2$; see the sentence after Notation~\ref{notation:leqP}. Write $m_1=t^{a_1}\chi^{u_1}$ and $m_2=t^{a_2}\chi^{u_2}$ and suppose that $m_1\leq_{Q}m_2$. Note that, because $v$ is the first row of a matrix defining $Q$, this implies that $\angbra{(a_1,u_1),v}\leq\angbra{(a_2,v_2),v}$. We consider three cases.

\begin{itemize}
\item Case 1: $|m_2|_P=0_{\kappa(P)}$, i.e., $m_2\in\idker(P)$. Then, for any $i$, $\angbra{(a_2,u_2),w_i}=-\infty$ so $\angbra{(a_2,u_2),v}<0$ by Lemma~\ref{lemma:VandW}(\ref{lemmaPart:infinite}). Thus $\angbra{(a_1,u_1),v}\leq\angbra{(a_2,u_2),v}<0$ so Lemma~\ref{lemma:VandW}(\ref{lemmaPart:infinite}) tells us that $\angbra{(a_1,u_1),w_i}=-\infty$ for all $i$, i.e., $|m_1|_P=0_{\kappa(P)}$. We have $|m_1|_P=0_{\kappa(P)}\leq|m_2|_P$. 

\item Case 2: $|m_1|_P=0_{\kappa(P)}$. Since $0_{\kappa(P)}$ is the smallest element of $\kappa(P)$, we have $|m_1|_P=0_{\kappa(P)}\leq|m_2|_{P}$.

\item Case 3: $|m_1|_P$ and $|m_2|_P$ are both not $0_{\kappa(P)}$. Since, for all $i$,  $\angbra{(a_1,u_1),w_i}\neq-\infty$ and $\angbra{(a_2,u_2),w_i}\neq-\infty$, Lemma~\ref{lemma:VandW}(\ref{lemmaPart:finite}) tells us that $\angbra{(a_1,u_1),v}=0=\angbra{(a_2,u_2),v}$, $\angbra{(a_1,u_1),\what{w}_i}=\angbra{(a_1,u_1),w_i}$, and $\angbra{(a_2,u_2),\what{w}_i}=\angbra{(a_2,u_2),w_i}$. So, because $Q$ is given by the matrix 
$\begin{pmatrix}v\\\what{w}_0\\\what{w}_1\\\vdots\\\what{w}_k\end{pmatrix}$
and $P$ is given by the matrix 
$\begin{pmatrix}w_0\\w_1\\\vdots\\w_k\end{pmatrix}$, the fact that $m_1\leq_{Q}m_2$ implies that $m_1\leq_P m_2$.
\end{itemize}
\par\nopagebreak\vspace{-1.5\baselineskip}\mbox{}
\end{proof}

\begin{remark}\label{rmk:forTropicalizedCases}
Intuitively, Theorem~\ref{thm:resolving-primes} says that, if no component of $\wt{V}(E)$ is contained in the toric boundary, then every minimal prime congruence of $\base[\Mon]/E$ has trivial kernel. 
In particular, suppose $K$ is a field with valuation $v:K\to\base$, let $X_\sigma$ denote the toric variety $\Spec K[\Mon]$, and let $T$ be the big torus in $X_\sigma$. 

Then, by Remark~\ref{rmk: Gubler-recession},
$$\wt{V}(E)=\cl_{\R_{\geq0}\times N_{\R}(\sigma)}\Big(  \R_{\geq0} \cdot \big( \{1\} \times \trop(Y) \big)  \Big).$$
Now, if $Y=V(I)$ has no component contained in the toric boundary $X_\sigma\sdrop T$, then by \cite[Theorem 6.2.18]{MS} we have 
$$\trop(Y)=\trop\Big(\cl_{X_\sigma}(Y\cap T)\Big)=\cl_{N_{\R}(\sigma)}\Big(\trop(Y\cap T)\Big)=\cl_{N_{\R}(\sigma)}\Big(\trop(Y)\cap N_{\R}\Big)$$ 
and so 
$\wt{V}(E)$ is the closure of $\wt{V}(E)\cap(\R_{\geq0}\times N_{\R})$ in $\R_{\geq0}\times N_{\R}(\sigma)$. Thus every minimal prime congruence of $\base[\Mon]/E=\base[\Mon]/\Bend(\trop I)$ has trivial ideal-kernel.
\end{remark}

\begin{remark}\label{rmk: Gubler-recession}
    Let $J \subseteq K[\Mon]$, where $K$ is a field with valuation $v:K\to\T$ and $\Mon$ and toric monoid. Let $I = \trop J \subseteq \T[\Mon]$ and so $V(I)=V(\Bend I) = \trop V(J)$. Then $\wt{V}(\Bend I) = \cl_{\R_{\geq0}\times N_{\R}(\sigma)}(\R_{\geq0} \cdot (\{1\} \times V(I))).$ 

    To see this, first note that $\wt{V}(\Bend I)=\Trop_W(V(J))$ where $\Trop_W(V(J))$ is as in \cite[Definition 8.3]{Gub11} and so we can assume that $V(J)$ is irreducible, since both closure and tropicalization respect finite unions.
    Next we can assume that $V(I) \cap N_{\R} \neq \emptyset$, since if $V(I) \cap N_{\R} = \emptyset$, then we can replace $N_{\R}(\sigma)$ with $N_\R(\sigma/\tau)$, where $\tau$ is the smallest face of $\sigma$ such that $V(I) \cap N_\R/\tau \neq \emptyset$.
    Now we note that $\wt{V}(\Bend I) \supseteq\cl_{\R_{\geq0}\times N_{\R}(\sigma)}(\R_{\geq0} \cdot (\{1\} \times V(I)))$ because $\wt{V}(\Bend I)$ is closed. On the other hand, 
\begin{align*}
    \wt{V}(\Bend I) &= \cl_{\R_{\geq0}\times N_{\R}(\sigma)}\left(\wt{V}(\Bend I) \cap (\R_{\geq0}\times N_\R)\right)\\
    &= \cl_{\R_{\geq0}\times N_{\R}(\sigma)}\bigg( \cl_{\R_{\geq0}\times N_{\R}} \Big( \R_{\geq0} \cdot \big(\{1\} \times (V(I)\cap N_\R) \big) \Big)  \bigg) \\
    &= \cl_{\R_{\geq0}\times N_{\R}(\sigma)}\Big(\R_{\geq0} \cdot \big(\{1\} \times (V(I)\cap N_\R) \big)  \Big) \\
    &\subseteq \cl_{\R_{\geq0}\times N_{\R}(\sigma)}\Big(\R_{\geq0} \cdot \big(\{1\} \times V(I)\big)\Big),
\end{align*}
    where the second equality follows from \cite[Corollary 11.13]{Gub11}.
\end{remark}

\begin{remark}\label{rmk:errorInMR}
    We suspect that, for any tropical ideal $I$ such that $V(I)=\cl_{N_{\R}(\sigma)}(V(I)\cap N_{\R})$, every minimal prime congruence of $\base[\Mon]/\Bend(\trop I)$ has trivial ideal kernel. 
    The proof in this case would rely on a statement corresponding to Remark~\ref{rmk: Gubler-recession} for tropical ideals. Such a statement is \cite[Proposition 2.12]{MR20}. However, the authors recently found a flaw in the proof of that result. 
\end{remark}

\begin{prop}\label{prop:RadicalOfDiagonalIsQCIfPrimesCanBeResolved}
Let $A$ be an additively idempotent semiring. If each prime congruence of $A$ contains a prime congruence with trivial ideal-kernel, then $A/\sqrt{\Delta}$ is cancellative.
\end{prop}\begin{proof}

By hypothesis, we can write $\sqrt{\Delta}=\dcap_{i\in I}P_i$ where each $P_i$ is a prime congruence of $A$ with trivial ideal-kernel. Consider $g\neq0$ in $A$; we will show that the image of $g$ is cancellative in $A/\sqrt{\Delta}$. 
So suppose $gf_1$ is equal to $gf_2$ modulo $\sqrt{\Delta}$; we want to show $(f_1,f_2)\in\sqrt{\Delta}$. Thus, for each $i\in I$, $|g|_{P_i}|f_1|_{P_i}=|g|_{P_i}|f_2|_{P_i}$. Since $P_i$ has trivial ideal-kernel, $|g|_{P_i}\neq0_{\kappa(P)}$, so $|f_1|_{P_i}=|f_2|_{P_i}$, i.e., $(f_1,f_2)\in P_i$, because $P_i$ is prime. Thus $(f_1,f_2)\in\dcap_{i\in I}P_i=\sqrt{\Delta}$, as desired.
\end{proof}

\begin{coro}\label{coro:ClosureConditionImpliesRadicalQC}
Let $\base$ be a sub-semifield of $\T$, let $\Mon$ be a toric monoid, and let $E$ be a congruence on $\base[\Mon]$ with a finite tropical basis. Suppose that $\wt{V}(E)$ is the closure of $\wt{V}(E)\cap(\R_{\geq0}\times N_{\R})$ in $\R_{\geq0}\times N_{\R}(\sigma)$. Then $\base[\Mon]/\sqrt{E}$ is cancellative.
\end{coro}
\begin{proof}
This follows from Theorem~\ref{thm:resolving-primes} and Proposition~\ref{prop:RadicalOfDiagonalIsQCIfPrimesCanBeResolved}.
\end{proof}

\section{Convex Piecewise-Linear Functions}\label{app:CPL}

In this section we provide a geometric proof of Corollary~\ref{coro:ClosureConditionImpliesRadicalQC} using the convex piece-wise linear functions on the variety of the congruence $E$. This proof does not tell us anything about the minimal primes over $E$ or how to get them.

This geometric approach allows us to view the quotient by the radical of a congruence as the functions on the variety of said congruence. 

In this section $\Lattice$ is a finitely generated free abelian group and $\Lattice_\R $ denotes the vector space $ \Lattice\otimes_{\Z}\R$, $N:=\Lattice^*=\Hom(\Lattice,\Z)$ and $N_\R=\Hom(\Lattice,\R)\cong N\otimes_{\Z}\R$. The cone $\sigma$ is a strongly convex rational polyhedral cone and $ \Mon=\Mon_{\sigma}$, the toric monoid corresponding to $\sigma$. We use $\tau$ to denote a face of $\sigma$.

\begin{defi}
A function $\ph:\R_{\geq0}\times N_{\R}\to \R\cup\{-\infty\}$ is called $(\Gamma,\Z)$-linear if it is of the form $\ph(r,x)=r\gamma+\angbra{x,u}$ for some $\gamma\in\Gamma$ and $u\in\Lattice$. 
\end{defi}

Note that such a $\ph$ extends continuously to a point $(r,x)\in\R_{\geq0}\times (N_{\R}/\tau)$ if and only if $u\in \Mon_{\tau}:=\tau^\vee\cap \Lattice$. Thus, when we consider the corresponding functions on $\R_{\geq0}\times N_{\R}(\sigma)$, we only allow $u\in\Mon$.

\begin{defi}
A function $\ph:\R_{\geq0}\times N_{\R}(\sigma)\to \R\cup\{-\infty\}$ is called $(\Gamma,\Z)$-linear if it is of the form $\ph(r,x)=r\gamma+\angbra{x,u}$ for some $\gamma\in\Gamma$ and $u\in\Mon$. A \emph{convex, piecewise-$(\Gamma,\Z)$-linear function} on $\R_{\geq0}\times N_{\R}(\sigma)$ is a convex piecewise-linear function $\R_{\geq0}\times N_{\R}(\sigma)\to\R\cup\{-\infty\}$ such that each of the linear functions used is $(\Gamma,\Z)$-linear. We let $\CPL(\R_{\geq0}\times N_{\R}(\sigma))=\CPL_{\Gamma}(\R_{\geq0}\times N_{\R}(\sigma))$ denote the semiring of such functions, where the operations are given by taking the maximum of two functions and taking the (real valued) sum. 
\end{defi}

Thus $\Phi:\R_{\geq0}\times N_{\R}(\sigma)\to\R\cup\{-\infty\}$ is in $\CPL_{\Gamma}(\R_{\geq0}\times N_{\R}(\sigma))$ exactly if it can be written as $\Phi=\max\{\ph_1,\ldots,\ph_n\}$ with each of $\ph_1,\ldots\ph_n$ being $(\Gamma,\Z)$-linear, i.e., if it is of the form $\Phi=\wt{f}$ for some $f\in \base[\Mon]$; see \cite[Exercise 1.9.4]{MS} and the introduction of \cite{TW24}\footnote{Here, as in Section~\ref{sec:Filters}, if $f=\dsum_{u\in\Mon} f_u\chi^u$, then
$\wt{f}(r,x)=\displaystyle\max_{u\in\Mon}\left(r\log(f_u)+_{\R}\angbra{x,u}\right)$ for all $(r,x)\in\R_{\geq0}\times N_{\R}(\sigma)$.}.

If $(r,x)\in \R_{\geq0}\times N_{\R}(\sigma)$ has $x\in N_{\R}/\tau$, then $\tau$ is the smallest face $\rho$ of $\sigma$ such that $(r,x)\in \R_{\geq0}\times N_{\R}(\rho)$. This motivates the following definition and shows that it agrees with the previous ones.

\begin{defi}
Let $X$ be any subset of $\R_{\geq0}\times N_{\R}(\sigma)$ and let $\tau$ be the smallest face of $\sigma$ such that $X\subseteq\R_{\geq0}\times N_{\R}(\tau)$. A \emph{convex, piecewise-$(\Gamma,\Z)$-linear function} on $X$ is a function $X\to \R\cup\{-\infty\}$ that is the restriction of some $\Phi\in\CPL_{\Gamma}(\R_{\geq0}\times N_{\R}(\tau))$. We let $\CPL_\Gamma(X)=\CPL(X)$ denote the semiring of such functions.
\end{defi}

In particular, in the above situation there is a surjective evaluation map $\base[\Mon_{\tau}]\to\CPL_{\Gamma}(X)$.

The following lemma allows us to make connections between congruences on $\base[\Mon]$ and $\base[\Mon_\tau]$ and consequently between the convex, piecewise-$(\Gamma,\Z)$-linear functions on their varieties.
\begin{lemma}
\label{lemma: extensionToMoreMonomials}
    Let $E$ be a congruence on $\base[\Mon]$ and let $\tau$ be a face of $\sigma$. Then the push-forward of $E$ along the inclusion $\base[\Mon]\into\base[\Mon_\tau]$ consists of pairs of the form $m\alpha$, where $m$ is a monomial of $\base[\Mon_\tau]$ and $\alpha\in E$.  
\end{lemma}
\begin{proof}
    By \cite[Lemma 2.4]{MR14}, the congruence $E'$ on $\base[\Mon_\tau]$ generated by $E$ is equal to the transitive closure of the set $U$ of all pairs of the form $m(a,b)+(h,h)$ with $(a,b)\in E$, $m,h\in\base[\Mon_\tau]$ and m is a monomial. Since $h\in\base[\Mon_\tau]$ there are $h'\in\base[\Mon]$ and $m'$ a monomial in $\base[\Mon_\tau]$ such that $h=m'h'$.

    Pick any $\mu\in(\relint(\sigma^\vee))\cap\Lattice\subseteq\Mon$, considered as a monomial of $\base[\Mon]$. Then, for any sufficiently large $n$ we have $m\cdot\mu^n,m'\cdot\mu^n\in\base[\Mon]$. So $m\mu^n(a,b)\in E$ and $\mu^n(h,h)=m'\mu^n(h',h')\in E$. Since $\mu^{-n}$ is a monomial of $\base[\Mon_\tau]$, we see that we can write every $m(a,b)+(h,h)=\mu^{-n}(m\mu^n(a,b)+\mu^n(h,h))$ in $U$ as a monomial of $\base[\Mon_\tau]$ times a pair in $E$.

    A similar argument shows that $U$ is already transitively closed, so $E'=U$.
\end{proof}

\begin{coro}\label{coro:cong_variety_of_pushforward}
Let $E$ be a congruence on $\base[\Mon]$, let $\tau$ be a face of $\sigma$, and let $E_\tau$ be the push-forward of $E$ along the inclusion $\base[\Mon]\into\base[\Mon_\tau]$. Then $\wt{V}(E_\tau) = \wt{V}(E) \cap \big(\R_{\geq0}\times N_{\R}(\tau)\big)$.
\end{coro}\begin{proof}
It suffices to show that, for any $w\in \R_{\geq0}\times N_{\R}(\tau)$, $w\in \wt{V}(E)$ if and only if $w\in\wt{V}(E_\tau)$. Given Lemma~\ref{lemma: extensionToMoreMonomials}, it is straightforward to check this.
\end{proof}

\begin{proposition}\label{prop: quotientByRadofBendIsFunctions}
Let $E$ be a congruence on $\base[\Mon]$ with a finite tropical basis. Then the evaluation map induces an injection
$$\frac{\base[\Mon]}{\sqrt{E}}\into\CoordinateFunctionSemiring\left(\wt{V}(E)\right).$$
Moreover, if $\tau$ is the smallest face of $\sigma$ such that $\wt{V}(E)\subseteq \R_{\geq0}\times N_{\R}(\tau)$ and we let $E_{\tau}$ be the push-forward of $E$ along the inclusion $\base[\Mon]\into\base[\Mon_{\tau}]$, then the evaluation map induces an isomorphism 
$$\frac{\base[\Mon_\tau]}{\sqrt{E_\tau}}\cong\CoordinateFunctionSemiring\left(\wt{V}(E)\right).$$
\end{proposition}\begin{proof}

Let $\operatorname{ev}_{\wt{V}(E)}:\base[\Mon]\to\CPL(\wt{V}(E))$ be the evaluation map.

We can view $\CoordinateFunctionSemiring\left(\wt{V}(E)\right)$ as a subsemiring of 
$\dprod_{x\in\wt{V}(E)}\T$, so the congruence-kernel of $\operatorname{ev}_{\wt{V}(E)}$ is the intersection of the congruence kernels of the evaluations maps $\operatorname{ev}_w:\base[\Mon]\to\T$ given by evaluation at $w$ for $w\in\wt{V}(E)$. But the congruence kernel of $\operatorname{ev}_w$ is $P_w$, so the congruence kernel of $\operatorname{ev}_{\wt{V}(E)}$ is $\dcap_{w\in\wt{V}(E)}P_w=\sqrt{E}$ by Corollary~\ref{coro:TropBasisImpliesRankOnePrimesEnough}. Thus $\operatorname{ev}_{\wt{V}(E)}$ induces an injective homomorphism $\dfrac{\base[\Mon]}{\sqrt{E}}\into\CoordinateFunctionSemiring\left(\wt{V}(E)\right)$,
completing the proof of the first part of the proposition.

By Corollary~\ref{coro:cong_variety_of_pushforward}, $\wt{V}(E)=\wt{V}(E_\tau)$. Moreover, any tropical basis of $E$ is also a tropical basis of $E_\tau$. So we may assume without loss of generality that $\tau=\sigma$.

In this case, we know that the evaluation map $\operatorname{ev}_{\wt{V}(E)}:\base[\Mon]\to\CPL(\wt{V}(E))$ is surjective. 
On the other hand, the first part of the proposition shows that $\operatorname{ev}_{\wt{V}(E)}$ has kernel $\sqrt{E}$. Thus $\operatorname{ev}_{\wt{V}(E)}$ induces an isomorphism $\dfrac{\base[\Mon]}{\sqrt{E}}\cong\CoordinateFunctionSemiring\left(\wt{V}(E)\right)$.
\end{proof}

The following two corollaries are immediate consequences of Proposition~\ref{prop: quotientByRadofBendIsFunctions}

\begin{coro}\label{coro:quotientIsPLFunctionsWithoutPushforward}
Let $E$ be a congruence on $\base[\Mon]$ with a finite tropical basis. Suppose that there is no smaller face $\tau\lneq\sigma$ such that $\R_{\geq0}\times N_{\R}(\tau)$ contains $\wt{V}(E)$.  Then the evaluation map induces an isomorphism 
$$\frac{\base[\Mon]}{\sqrt{E}}\cong\CoordinateFunctionSemiring\left(\wt{V}(E)\right).$$
\end{coro}

\begin{coro}
Let $E$ be a congruence on $\base[x_1^{\pm1},\ldots,x_n^{\pm1}]$ with a finite tropical basis. Then the evaluation map induces an isomorphism 
$$\frac{\base[x_1^{\pm1},\ldots,x_n^{\pm1}]}{\sqrt{E}}\cong\CoordinateFunctionSemiring\left(\wt{V}(E)\right).$$
\end{coro}

On the other hand, we can also use Proposition~\ref{prop: quotientByRadofBendIsFunctions} to give another proof of Theorem~\ref{thm-intro-rad}, which we restate as Corollary~\ref{coro: IntroThmAsAppendixCoro}

\begin{coro}\label{coro: IntroThmAsAppendixCoro}
Let $\base$ be a sub-semifield of $\T$, let $\Mon$ be the toric monoid corresponding to the cone $\sigma$, and let $E$ be a congruence on $\base[\Mon]$ with a finite tropical basis. Suppose that $\wt{V}(E)$ is the closure of $\wt{V}(E)\cap(\R_{\geq0}\times N_{\R})$ in $\R_{\geq0}\times N_{\R}(\sigma)$. Then $\base[\Mon]/\sqrt{E}$ is cancellative.
\end{coro}\begin{proof}
By Proposition~\ref{prop: quotientByRadofBendIsFunctions}, we have an injective homomorphism $\dfrac{\base[\Mon]}{\sqrt{E}}\into\CPL(\wt{V}(E))$.

Now consider the restriction map $\CPL\left(\wt{V}(E)\right)\to \CPL\left(\wt{V}(E)\cap(\R_{\geq0}\times N_{\R})\right)$. Since $\wt{V}(E)$ is the closure in $\R_{\geq0}\times N_{\R}(\sigma)$ of $\wt{V}(E)\cap(\R_{\geq0}\times N_{\R})$ and all functions in $\CPL\left(\wt{V}(E)\right)$ are continuous, the restriction map $\CPL\left(\wt{V}(E)\right)\to \CPL\left(\wt{V}(E)\cap(\R_{\geq0}\times N_{\R})\right)$ is injective.

We now have an injective homomorphism $$\dfrac{\base[\Mon]}{\sqrt{E}}\into\CPL\left(\wt{V}(E)\right)\into\CPL\left(\wt{V}(E)\cap(\R_{\geq0}\times N_{\R})\right).$$ Since all elements of $\CPL\left(\wt{V}(E)\cap(\R_{\geq0}\times N_{\R})\right)$ are real-valued functions except for the zero element of $\CPL\left(\wt{V}(E)\cap(\R_{\geq0}\times N_{\R})\right)$, i.e., the function that is identically $-\infty$, they are cancellative. Thus $\CPL\left(\wt{V}(E)\cap(\R_{\geq0}\times N_{\R})\right)$ is cancellative and so is $\dfrac{\base[\Mon]}{\sqrt{E}}$. %is cancellative as well. 
\end{proof}

We can also leverage these to get results for classical congruence varieties.

We define convex piecewise $(\Gamma,\Z)$-affine functions on $N_{\R}(\sigma)$ (or $X\subseteq N_{\R}(\sigma)$) in direct analogy with convex piecewise $(\Gamma,\Z)$-linear functions on $\R_{\geq0}\times N_{\R}(\sigma)$. Here a function $N_{\R}(\sigma)\to\R\cup\{-\infty\}$ is $(\Gamma,\Z)$ affine if it is of the form $x\mapsto\gamma+\angbra{x,u}$ for some $u\in\Mon$. We let $\CPA(X)=\CPA_\Gamma(X)$ denote the \emph{semiring of convex piecewise $(\Gamma,\Z)$-affine functions} on $X$.

\begin{coro}\label{coro: quotientByRadofBendIsAffineFunctions}
Let $E$ be a congruence on $\base[\Mon]$ with a finite tropical basis. Suppose that the recession $\rec(V(E))$ of $V(E)$ is equal to $V((\pi_{\B})_*(E))$, where $\pi_{\B}:\base[\Mon]\to\B[\Mon]$ is the semiring homomorphism induced by the unique homomorphism $\base\to\B$. 
Then the evaluation map $\base[\Mon]\to\CPA\big(V(E)\big)$ induces an injection
$$\dfrac{\base[\Mon]}{\sqrt{E}}\into\CPA\big(V(E)\big).$$

Moreover, if there is no proper face $\tau$ of $\sigma$ such that $V(E)\subseteq N_{\R}(\tau)$. 
Then the evaluation map $\base[\Mon]\to\CPA\big(V(E)\big)$ induces an isomorphism
$$\dfrac{\base[\Mon]}{\sqrt{E}}\cong\CPA\big(V(E)\big).$$
\end{coro}\begin{proof}

By identifying $V(E)$ with $\{1\}\times V(E)\subseteq \wt{V}(E)$, we get a restriction map $r:\CPL(\wt{V}(E))\to\CPA(V(E))$. By definition, this map is surjective.

Note that the condition that $\rec(V(E))=V((\pi_{\B})_*(E))$ is equivalent to saying that $\wt{V}(E)$ is the closure in $\R_{\geq0}\times N_{\R}(\sigma)$ of $\R_{\geq0}\cdot (\{1\}\times V(E))$. Since any $\ph\in\CPL(\wt{V}(E))$ is continuous, we get that the restriction map $r:\CPL(\wt{V}(E))\to\CPA(V(E))$ is injective. Thus $r$ is an isomorphism.

Since $\operatorname{ev}_{V(E)}$ factors as $\base[\Mon]\toup{\operatorname{ev}_{\wt{V}(E)}}\CPL\left(\wt{V}(E)\right)\toup{r}\CPA\big(V(E)\big)$ and Proposition~\ref{prop: quotientByRadofBendIsFunctions} tells us that $\operatorname{ev}_{\wt{V}(E)}$ induces an injection $\dfrac{\base[\Mon]}{\sqrt{E}}\into\CPL(\wt{V}(E))$, we now see that $\operatorname{ev}_{V(E)}$ induces an injection $\dfrac{\base[\Mon]}{\sqrt{E}}\into\CPA(V(E))$.

Now suppose $V(E)$ is not contained in $N_{\R}(\tau)$ for any $\tau\lneq\sigma$. We get that $\wt{V}(E)$ is not contained in $\R_{\geq0}\times N_{\R}(\tau)$ for any $\tau\lneq\sigma$, so Corollary~\ref{coro:quotientIsPLFunctionsWithoutPushforward} tells us that the evaluation map $\operatorname{ev}_{\wt{V}(E)}$ induces an isomorphism $\base[\Mon]/\sqrt{E}\cong\CPL\left(\wt{V}(E)\right)$. Since we already know that $r$ is an isomorphism, we conclude that the map induced by $\operatorname{ev}_{V(E)}$ is an isomorphism $\dfrac{\base[\Mon]}{\sqrt{E}}\cong\CPA(\wt{V}(E))$

Since $\operatorname{ev}_{V(E)}$ factors as $\base[\Mon]\toup{\operatorname{ev}_{\wt{V}(E)}}\CPL\left(\wt{V}(E)\right)\toup{r}\CPA\big(V(E)\big)$, we now see that $\operatorname{ev}_{V(E)}$ induces an isomorphism $\dfrac{\base[\Mon]}{\sqrt{E}}\cong\CPA(V(E))$.
\end{proof}

\begin{coro}
Let $E$ be a congruence on $\base[x_1^{\pm1},\ldots,x_n^{\pm1}]$ with a finite tropical basis. Suppose that $\rec(V(E)) = V((\pi_{\B})_*(E))$. 
Then the evaluation map $\base[x_1^{\pm1},\ldots,x_n^{\pm1}]\to\CPA\big(V(E)\big)$ induces an isomorphism
$$\dfrac{\base[x_1^{\pm1},\ldots,x_n^{\pm1}]}{\sqrt{E}}\cong\CPA\big(V(E)\big).$$
\end{coro}\begin{proof}
This follows immediately from Corollary~\ref{coro: quotientByRadofBendIsAffineFunctions}.
\end{proof}

\begin{coro}\label{coro: functions-know-radical}
Let $J\subseteq k[\Mon]$ be an ideal such that there is no affine toric variety that is a proper toric subvariety of $X(\sigma)$ and which contains $V(J)$\footnote{In particular, this is true if $V(J)$ meets every boundary divisor.}.
Let $I=\trop(J)\subseteq \base[\Mon]$. Then the evaluation map $\base[\Mon]\to\CPA\big(V(\Bend(I))\big)$ induces an isomorphism
$$\dfrac{\base[\Mon]}{\sqrt{\Bend(I)}}\cong\CPA\big(V(\Bend I)\big).$$
\end{coro}\begin{proof}

Remark~\ref{rmk: Gubler-recession} guarantees that $\rec(V(\Bend I))=V((\pi_{\B})_*(\Bend I))$. Since the hypothesis on $V(J)$ guarantees that there is no proper face $\tau$ of $\sigma$ such that $V(\Bend I)\subseteq N_{\R}(\tau)$, the desired result follows from Corollary~\ref{coro: quotientByRadofBendIsAffineFunctions}.
\end{proof}

\begin{coro}\label{coro: functions-know-radical-torus}
Let $I\subseteq\base[x_1^{\pm1},\ldots,x_n^{\pm}]$ be the tropicalization of any ideal of $k[x_1^{\pm1},\ldots,x_n^{\pm1}]$. 
Then the evaluation map $\base[x_1^{\pm1},\ldots,x_n^{\pm}]\to\CPA\big(V(\Bend(I))\big)$ induces an isomorphism
$$\dfrac{\base[x_1^{\pm1},\ldots,x_n^{\pm}]}{\sqrt{\Bend(I)}}\cong\CPA\big(V(\Bend(I))\big).$$ 
\end{coro}

\begin{example}\label{ex: inf-trop-basis}
    Consider the prime congruences $P_1$ and $P_2$ on $\T[x]$ with defining matrices $\begin{pmatrix} 1 & 0 \end{pmatrix}$ and $\begin{pmatrix} 1 & 0 \\ 0 & -1 \end{pmatrix}$ respectively. The prime $P_2$ has no finite tropical basis by \cite[Example 6.3.4 and Example 7.1.9]{M16}. Note that $V(P_1) = V(P_2) = \{ 0_\R\}$, however, $\sqrt{P_1} \neq \sqrt{P_2}$ as prime congruences are also radical by \cite[Proposition 3.12]{JM17} and $P_1 \neq P_2$. \exEnd 
\end{example}

\appendix
\section{}\label{app:Eu-Ja}

\newcommand{\Rat}{\operatorname{Rat}}

In a series of papers \cite{SN24}, \cite{Son24}, \cite{Son23a}, \cite{Son23b} J.\ Song, et.\ al., defines the semifield of tropical rational functions on an abstract tropical curve, which is then extended in \cite{AKS25} to more general polyhedral complexes. They show that, if $\Pi\subseteq\R^n$ is a polyhedral complex in $\R^n$, then the tropical function semifield $\Rat(|\Pi|)$ can be found as follows. Let $R$ be the semiring of tropical Laurent polynomial functions on $\R^n$ (by \cite[Proposition~5.5]{JM17}, this is $\T[x_1^{\pm1}, \ldots, x_n^{\pm1}]/\sqrt{\Delta}$), and let $F$ be the semifield of fractions of $R$. Then $\Rat(|\Pi|)$ is the quotient of $F$ determined by restricting functions on $\R^n$ to $|\Pi|$. That is, $\Rat(|\Pi|)=F/E(|\Pi|)$ where, for any $V\subseteq\R^n$, $E(V)= \{(f,g) \in F^2 \,:\, f(w) = g(w), \forall w \in V\}$.

J.\ Song proves structure results for 
%$F/E(V)$ 
$\Rat(V)$ 
and establishes certain relations to corresponding tropical curves. She shows that, for a tropical curve $\Gamma$, the geometric structure of $\Gamma$ as a tropical curve is determined by the $\T$-algebra structure of $\text{Rat}(\Gamma)$. 

Our results 
%\green{in this paper}
allow us to relate Song's rational function semifield to bend congruences.

\begin{prop}
Let $k$ be a valued field, let $I\subseteq k[x_1^{\pm1},\ldots,x_n^{\pm1}]$ be an ideal, Let $X=V(I)$, $Y=\Trop(X)$, and $\operatorname{Rat}(Y)$ denote the tropical rational function semifield of $Y$ as per \cite{SN24}. Let $C=\Bend(\trop I)$. Then $\operatorname{Rat}(Y)\cong \Frac\left(\T[x_1^{\pm1},\ldots,x_n^{\pm1}]\big/\sqrt{C}\right)$, i.e., $$\operatorname{Rat}(\Trop V(I))\cong \Frac\left(\T[x_1^{\pm1},\ldots,x_n^{\pm1}]\Big/\sqrt{\Bend(\trop I)}\right).$$
\end{prop}
\begin{proof}

Let $E_{\pm}(Y):=\{(f,g)\in (\T[x_1^{\pm1},\ldots,x_n^{\pm1}])^2 \,:\, f(y)=g(y) \text{ for all }y\in Y\}$. Then, by the comment between Lemmas 3.28 and 3.29 of \cite{SN24}, we have 
$\Rat(Y)=\Frac\left(\T[x_1^{\pm1},\ldots,x_n^{\pm1}]/E_{\pm}(Y)\right)$.
But by Corollary~\ref{coro: functions-know-radical-torus}, $E_{\pm}(Y)$, being the congruence kernel of the evaluation map $\T[x_1^{\pm1},\ldots,x_n^{\pm1}]\to\CPA(Y)$, is equal to $\sqrt{C}$.
\end{proof}

%------------------------------------------------------------------

%====================================================================
%\newpage
%\bibliography{schemes}\bibliographystyle{alpha}
\bibliographystyle{alpha}

\end{document}